%% file: sdgft.tex
\documentclass[11pt,a4paper]{article}

\input{sdgft-head.tex}
\newcommand{\applink}{\autoref{sec:ptc}}

\title{\thetitle\footnotetext{\emph{Acknowledgements:} \acks.}\footnotetext{\emph{MSC 2010:} \themsc.}
}
\author{\fullname\\\footnotesize \addressone\\\footnotesize \addresstwo}
\date{}

\begin{document}

\maketitle

\begin{abstract}
  \input{sdgft-abstract.tex}
\end{abstract}

\input{sdgft-body.tex}

\subsection{Technical proofs}
\label{sec:ptc}

\input{sdgft-proofs.tex}

\bibliographystyle{abbrvnat}
{\footnotesize \bibliography{sdgft}}

\end{document}

%% file: sdgft-head.tex
\newcommand{\thetitle}{Semimartingale detection and goodness-of-fit tests}

\newcommand{\forenames}{Adam D.}  \newcommand{\surname}{Bull}
\newcommand{\fullname}{\forenames\ \surname} \newcommand{\acks}{We
  thank EPSRC for their support under grant EP/K000993/1.  All
  research data was randomly generated using software given in
  \citet{bull_software_2016}}

\newcommand{\mscone}{62M02}
\newcommand{\msctwo}{60J60}
\newcommand{\mscthree}{60J75}
\newcommand{\mscfour}{62G10}
\newcommand{\mscfive}{65T60}
\newcommand{\themsc}{
\mscone\ (primary); \msctwo, \mscthree, \mscfour, \mscfive\ (secondary)}

\newcommand{\kwdone}{diffusion}
\newcommand{\kwdtwo}{goodness-of-fit}
\newcommand{\kwdthree}{jump process}
\newcommand{\kwdfour}{semimartingale}
\newcommand{\kwdfive}{wavelets}
\newcommand{\thekeywords}{\kwdone, \kwdtwo, \kwdthree, \kwdfour, \kwdfive}

\newcommand{\addressone}{Statistical Laboratory}%
\newcommand{\addresstwo}{University of Cambridge}%

\usepackage[round]{natbib}
\usepackage{amsmath, amsfonts, amssymb, amsthm, mathtools, 
  booktabs, mparhack, enumitem, tikz, multirow}
\usetikzlibrary{external}
\usepackage[bookmarks, breaklinks, colorlinks, linkcolor=blue,
citecolor=blue]{hyperref}
\hypersetup{pdftitle={\thetitle}, pdfauthor={\fullname}}

\let\oldmarginpar\marginpar
\renewcommand\marginpar[1]{\-\oldmarginpar[\raggedleft\footnotesize 
#1]{\raggedright\footnotesize #1}}

\DeclarePairedDelimiter{\abs}{\lvert}{\rvert}
\DeclarePairedDelimiter{\norm}{\lVert}{\rVert}
\newcommand{\N}{\mathbb{N}}

\newcommand{\R}{\mathbb{R}}

\renewcommand{\P}{\mathbb{P}}
\newcommand{\E}{\mathbb{E}}
\newcommand{\F}{\mathcal{F}}
\newcommand{\Var}{\mathbb{V}\mathrm{ar}}
\newcommand{\Cov}{\mathbb{C}\mathrm{ov}}

\newtheorem{theorem}{Theorem}
\newtheorem{lemma}{Lemma}

\newtheorem{assumption}{Assumption}

\makeatletter
\renewcommand{\p@enumii}{}
\makeatother

\newcommand{\tint}{\begingroup\textstyle\int\endgroup}
\newcommand{\tsum}{\begingroup\textstyle\sum\endgroup}
\newcommand{\tprod}{\begingroup\textstyle\prod\endgroup}
\newcommand{\tmax}{\begingroup\textstyle\max\endgroup}
\newcommand{\tmin}{\begingroup\textstyle\min\endgroup}
\newcommand{\tsup}{\begingroup\textstyle\sup\endgroup}

\newcommand{\mel}{\MoveEqLeft}
\newcommand{\bal}{\!\begin{aligned}[t]\mel[1]\mathopen{}}
\newcommand{\eal}{\end{aligned}}


%% file: sdgft-abstract.tex
In quantitative finance, we often fit a parametric semimartingale
model to asset prices.  To ensure our model is correct, we must then
perform goodness-of-fit tests. In this paper, we give a new
goodness-of-fit test for volatility-like processes, which is easily
applied to a variety of semimartingale models. In each case, we reduce
the problem to the detection of a semimartingale observed under
noise. In this setting, we then describe a wavelet-thresholding test,
which obtains adaptive and near-optimal detection rates.


%% file: sdgft-body.tex
\section{Introduction}
\label{sec:int}

In quantitative finance, we often model asset prices as
semimartingales; in other words, we assume prices are given by a sum
of drift, diffusion and jump processes. As these models can be
difficult to fit to data, we often restrict our attention to a
parametric class, of which many have been suggested by
practitioners. To verify our choice of parametric class, we must then
perform goodness-of-fit tests.

As semimartingale models can be quite complex, there are many
potential tests to perform. In the following, we will be interested in
testing whether models accurately describe processes such as the
volatility, covolatility, vol-of-vol or leverage.  We will further be
looking for tests which can be shown to obtain good rates of detection
against a variety of alternatives.

While many goodness-of-fit tests exist in the literature, fewer have
been shown to obtain good detection rates. Those tests which do
achieve good rates are generally designed for one type of
semimartingale model, and one way of measuring performance.

In the following, we will therefore describe a new goodness-of-fit
test for volatility-like processes in semimartingales. Our test can
easily be applied to a wide range of models, including stochastic
volatility, jumps and microstructure noise, and obtains good detection
rates against both local and nonparametric alternatives.

Our method involves reducing any goodness-of-fit test to one of
semimartingale detection: given a series of observations, is the
series white noise, or does it contain a hidden semimartingale? We
will show how this problem can be solved efficiently, obtaining
adaptive and near-optimal detection rates.

We now describe in more detail the problems we consider, as well as
relevant previous work. Our goal will be to test the goodness-of-fit
of a parametric semimartingale model. Many such models have been
described, including simple models such as Black-Scholes or
Cox-Ingersoll-Ross; L\'evy models such as the generalised hyperbolic
or CGMY processes; and stochastic volatility models such as the Heston
or Bates models. (For definitions, see \citealp{cont_financial_2004,
  papapantoleon_introduction_2008}.)

In the simplest case, where our observations are known to come from a
stationary or ergodic diffusion process, a great many authors have
described goodness-of-fit tests. We briefly mention some initial work
\citep{ait-sahalia_testing_1996, corradi_specification_1999,
  kleinow_testing_2002} as well as more recent discussion
\citep{gonzalez-manteiga_updated_2013,
  papanicolaou_variation-based_2014, chen_asymptotically_2015}.

In a financial setting, however, even if our model is stationary, we
may need to test it against non-stationary alternatives. When
observations can come from a non-stationary diffusion, goodness-of-fit
tests have been described using the integrated volatility
\citep{corradi_specification_1999}, estimated residuals
\citep{lee_bickelrosenblatt_2006, lee_residual_2008,
  nguyen_time-changed_2010} and marginal density
\citep{ait-sahalia_stationarity-based_2012}. Goodness-of-fit tests
also exist for regressions between diffusions
\citep{mykland_anova_2006}.

In the following, we will be interested in goodness-of-fit tests which
not only detect non-stationary alternatives, but also achieve good
detection rates.  In this setting, \citet{dette_test_2003} propose a
test which can detect misspecification of the volatility at a rate
\(n^{-1/4}\) in \(L^2\) norm \citep[see also][]{dette_estimation_2006,
  podolskij_range-based_2008, papanicolaou_variation-based_2014}.

A similar test proposed by \citet{dette_testing_2008} detects
alternatives in a fixed direction at the faster rate \(n^{-1/2},\)
although the authors do not give rates in \(L^p.\) This test can also
be applied to more complex models, including stochastic volatility
\citep{vetter_estimation_2012} and microstructure noise
\citep{vetter_model_2012}.

In some volatility testing problems, previous work has described tests
which achieve optimal detection rates against nonparametric alternatives
\citep{reis_nonparametric_2014, bibinger_nonparametric_2015}. However,
these tests are specific to the problems considered, and do not assess
the goodness-of-fit of general models.

In the following, we will therefore describe a new method of
goodness-of-fit testing for volatility-like processes. We will show
how our approach applies to a wide variety of semimartingale models,
including those with jumps, stochastic volatility and microstructure
noise. In each case, we will obtain adaptive detection rates, with
near-optimal behaviour not only against alternatives in a fixed
direction, but also against nonparametric alternatives.

To construct our tests, we will reduce each goodness-of-fit problem to
one of semimartingale detection: we will construct a series of
observations \(Z_i,\) which under the null hypothesis are
approximately white noise, and then test whether the \(Z_i\) contain a
hidden semimartingale \(S_t.\)

For example, suppose we have a semimartingale
\[dX_t = b_t\,dt + \sqrt{\mu_t}\,dB_t,\]
where \(B_t\) is a Brownian motion, \(b_t\) and \(\mu_t\) are
predictable processes, and we make observations \(X_{t_i},\) \(i = 0,
\dots, n,\) where the times \(t_i \coloneqq i/n.\) Further suppose we
have a model \(\mu(t, X_t)\) for the volatility, and wish to test the
hypotheses
\[H_0: \mu_t = \mu(t, X_t) \qquad \mathrm{vs.} \qquad H_1: \mu_t\
\mathrm{unrestricted.}\]

To estimate \(\mu_{t},\) we define the realised volatility estimates
\[Y_i \coloneqq n(X_{t_{i+1}}-X_{t_i})^2,\qquad i = 0, \dots, n-1.\]
Since the scaled increments \(\sqrt{n}(X_{t_{i+1}} - X_{t_i})\) are
approximately \(N(0, \mu_{t_i}),\) the observations \(Y_i\) have
approximate mean \(\mu_{t_i}\) and variance \(2\mu_{t_i}^2.\) Under
\(H_0,\) we thus have that the normalised observations
\[Z_i \coloneqq (Y_i - \mu(t_i, X_{t_i}))/\sigma(t_i, X_{t_i}), \qquad
\sigma^2 \coloneqq 2\mu^2,\] are approximately white noise.

Under \(H_1,\) we instead obtain
\begin{equation}
\label{eq:smp}
Z_i = S_{t_i} + \varepsilon_i,
\end{equation}
where the semimartingale
\[S_t \coloneqq (\mu_{t} - \mu(t, X_{t}))/\sigma(t, X_{t}),\]
and the approximately-centred noises
\[\varepsilon_i \coloneqq (Y_i - \mu_{t_i})/\sigma(t_i, X_{t_i}).\]
To test our hypotheses, we must therefore test whether the the series
\(Z_i\) is approximately white noise, or contains a hidden
semimartingale \(S_t.\)

If the noises \(\varepsilon_i\) were independent standard Gaussian,
independent of \(S_t,\) we could consider this a standard detection
problem in nonparametric regression. Conditioning on \(S_t,\) we could
take the semimartingale as fixed, and then apply the methods of
\citet{ingster_nonparametric_2003}, for example.

Under suitable assumptions on the process \(S_t,\) its sample paths
would be almost \(\tfrac12\)-smooth, and we would thus be able to
detect a signal \(S_t\) at rate \(n^{-1/4}\) in supremum norm, up to
log terms. Alternatively, if we wished to detect signals \(S_t \propto
e_t,\) for a fixed direction \(e_t,\) we could do so at a rate
\(n^{-1/2}.\)

In general, however, the signal \(S_t\) may depend on past values of
the noises \(\varepsilon_i,\) and vice versa. We will thus not be able
to appeal directly to results in nonparametric regression, and will
instead need to use arguments developed specifically for the
semimartingale setting.

In the following, we will show that testing problems
like~\eqref{eq:smp} can be solved with detection rates similar to
those of nonparametric regression. We will further show that many
semimartingale goodness-of-fit tests can be described in a form
like~\eqref{eq:smp}, including models with stochastic volatility,
jumps or microstructure noise.

Our approach will be similar to wavelet thresholding
\citep{donoho_wavelet_1995, hoffmann_adaptive_2010}; essentially, we
will reject the null whenever a suitable wavelet-thresholding estimate
of \(S_t\) is non-zero. While this method is known to work well in the
standard nonparametric setting, we will need to prove new results to
apply it to settings like~\eqref{eq:smp}.

Our proofs will use a Gaussian coupling derived from Skorokhod
embeddings. We note that as our results must apply in a general
semimartingale setting, we will not be able to use faster-converging
couplings, such as the KMT approximation. We will show, however, that
under reasonable moment bounds, a Skorokhod embedding will suffice to
achieve the desired detection rates.

Indeed, with this construction we will show our tests detect
semimartingales \(S_t\) at a rate \(n^{-1/4}\) in supremum norm, up to
log terms, even when \(S_t\) contains finite-variation
jumps. Furthermore, our tests will simultaneously detect simpler
signals at faster rates; for example, we will be able to detect
signals \(S_t\) in a fixed direction \(e_t\) at a rate \(n^{-1/2}\) up
to logs, without knowledge of the direction \(e_t.\)

We will finally show that in each case, the rates obtained are
near-optimal.  Applying our tests to problems like~\eqref{eq:smp}, we
will thus be able to construct goodness-of-fit tests for a wide
variety of semimartingale models, obtaining adaptive and near-optimal
detection rates.

The paper will be organised as follows. In \autoref{sec:det}, we give
a rigorous description of the problems we consider, and discuss
examples. In \autoref{sec:tes}, we then construct our tests, and state
our theoretical results. In \autoref{sec:emp}, we then give empirical
results, and in \autoref{sec:pro}, proofs.

\section{Semimartingale detection problems}
\label{sec:det}

We now describe our concept of a semimartingale detection problem. Our
setting will include volatility goodness-of-fit problems
like~\eqref{eq:smp}, as well as many other semimartingale
goodness-of-fit tests.

We begin with some examples of the problems we will consider. In each
case, we will describe a semimartingale model with a volatility-like
process \(\mu_t.\) We will wish to test the null hypothesis that
\(\mu_t\) is given by some known function \(\mu(\theta_0, t, X_t),\)
for an unknown paramter \(\theta_0 \in \Theta,\) and an estimable
covariate process \(X_t \in \R^q;\) our alternative hypothesis will be
that \(\mu_t\) is not given by \(\mu.\)

To test our hypothesis, we will construct \(\mathcal
F_{t_{i+1}}\)-measurable observations \(Y_i,\) and a variance function
\(\sigma^2.\) Under the null, and conditional on \(\mathcal F_{t_i},\)
the observations \(Y_i\) will have approximate mean and variance
\(\mu(\theta_0, t_i, X_{t_i})\) and \(\sigma^2(\theta_0, t_i,
X_{t_i}).\) To estimate these means and variances, we will further
construct estimates \(\widehat \theta\) and \(\widehat X_{i}\) of the
parameters \(\theta_0\) and covariates \(X_{t_i}.\)

We will then be able to estimate the difference between the
observations \(Y_i\) and their means \(\mu,\) scaled according to
their variances \(\sigma^2;\) we will reject the null hypothesis when
the size of these scaled differences are large. In \autoref{sec:tes},
we describe in detail how we perform such tests, as well as giving
theoretical results on their performance.

For now, we proceed with some examples of semimartingale
goodness-of-fit problems in this form.  Let \(B_t\) and \(B_t'\) be
independent Brownian motions, \(\lambda(dx, dt)\) be an independent
Poisson random measure with intensity \(dx\,dt,\) \(b_t\) and \(b_t'\)
be predictable locally-bounded processes, and \(f_t(x)\) be a
predictable function with \(\tint_\R 1 \wedge \abs{f_t(x)}^\beta\,dx\)
locally bounded, for some \(\beta \in [0, 1).\) Further define times
\(t'_i \coloneqq i/n^2.\)

We then have the following examples.

\begin{description}
\item[Local volatility] We wish to test a model \(\mu\) for \(\mu_t\)
  in the process
  \begin{equation}
    \label{eq:lv}
    dX_t = b_t\,dt + \sqrt{\mu_t}\,dB_t,
  \end{equation}
  making observations \(X_{t_i},\) \(i = 0, \dots, n.\) We set
  \(\widehat X_{i} \coloneqq X_{t_i},\) and estimate \(\mu_{t_i}\) by
  the realised volatility \citep{andersen_distribution_2001,
    barndorff-nielsen_econometric_2002},
  \[Y_i \coloneqq n(X_{t_{i+1}}-X_{t_i})^2.\]
  We then define the variance function
  \(\sigma^2 \coloneqq 2\mu^2.\)

\item[Jumps]
  We wish to test a model \(\mu\) for \(\mu_t\) in the process
  \[dX_t = b_t\,dt + \sqrt{\mu_t}\,dB_t + \tint_\R f_t(x)
  \,\lambda(dx,dt),\]
  making observations \(X_{t_i},\) \(i = 0, \dots, n.\) We set
  \(\widehat X_{i} \coloneqq X_{t_i},\) and estimate \(\mu_{t_i}\) by
  the truncated realised volatility \citep{mancini_non_2006,
    jacod_remark_2012},
  \[Y_i = g_n(\sqrt{n}(X_{t_{i+1}} - X_{t_i})), \qquad g_n(x) = x^2
    1_{x^2 < \alpha_n},\] for any sequence \(\alpha_n > 0\) satisfying
  \begin{equation}
    \label{eq:an}
    \log(n) = o(\alpha_n),\qquad \alpha_n = o(n^\kappa) \text{ for all }
    \kappa > 0.
  \end{equation}
  We then define the variance function
  \(\sigma^2 \coloneqq 2\mu^2.\)

\item[Microstructure noise]
  We wish to test a model \(\mu\) for \(\mu_t\) in the process
  \[dX_{1,t} = b_t\,dt + \sqrt{\mu_t}\,dB_t.\]
  We make observations
  \[\widetilde X_{1,i} \coloneqq X_{1,t_i'} + \varepsilon_i,\qquad i =
    0, \dots, n^2,\] where the noises \(\varepsilon_i\) are measurable
  in the filtrations \(\F_{t_i'}^+ \coloneqq \bigcap_{s > {t_i'}} \F_s,\) and
  satisfy
  \begin{align*}
    \E[\varepsilon_i \mid \mathcal F_{t_i'}] &= 0, \\
    \E[\varepsilon_i^2 \mid \mathcal F_{t_i'}] &= X_{2,t_i'},\\
    \E[\abs{\varepsilon_i}^{\kappa}\mid\mathcal F_{t_i'}] &\le C,
  \end{align*}
  for an It\=o semimartingale \(X_{2,t}\) with locally-bounded
  characteristics, and constants \(\kappa > 8,\) \(C > 0.\) We
  estimate \(X_{t_j}\) and \(\mu_{t_j}\) by their pre-averaged
  counterparts \citep{jacod_microstructure_2009,
    reis_asymptotic_2011},
  \begin{align*}
    \widehat X_{1,j}
    &\coloneqq n^{-1} \tsum_{i=0}^{n-1} \widetilde X_{1,nj+i},\\
    \widehat X_{2,j}
    &\coloneqq (2n)^{-1}\tsum_{i=0}^{n-1} (\widetilde X_{1,nj+i+1}
      - \widetilde X_{1,nj + i})^2,\\
    Y_j
    &\coloneqq \pi^2(2n^{-1}(\tsum_{i=0}^{n-1} \cos(\pi (i+\tfrac12)/n)\widetilde
      X_{1,nj+i})^2 - \widehat X_{2,j}).
  \end{align*}
  We then define the variance function \(\sigma^2 \coloneqq 2(\mu +
  \pi^2X_{2,t})^2.\)

\item[Stochastic volatility]
  We wish to test a model \(\mu\) for \(\mu_t\) in the processes
  \begin{align*}
    dX_{1,t} &= b_t\,dt + \sqrt{X_{2,t}}\,dB_t,\\
    dX_{2,t} &= b'_t\,dt + \sqrt{\mu_t}\,dB'_t,    
  \end{align*}
  making observations \(X_{1,t_i'},\) \(i = 0, \dots, n^2.\) We define
  volatility estimates
  \[\widetilde X_{2,i} \coloneqq n^2(X_{1,t_{i+1}'} - X_{1,t_i'})^2, \qquad
  i = 0, \dots, n^2-1,\]
  which we use to estimate \(X_{t_j}\) and \(\mu_{t_j}\)
  \citep{barndorff-nielsen_stochastic_2009, vetter_estimation_2012},
  \begin{align*}
    \widehat X_{1,j} &\coloneqq X_{1,t_j},\\
    \widehat X_{2,j}
    &\coloneqq n^{-1} \tsum_{i=0}^{n-1} \widetilde X_{2,nj+i},\\
    Y_j
    &\coloneqq 2\pi^2(n^{-1}(\tsum_{i=0}^{n-1} \cos(\pi (i+\tfrac12)/n) \widetilde
      X_{2,nj+i})^2
      - \widehat X_{2,j}^2).
  \end{align*}
  We then define the variance function \(\sigma^2 \coloneqq 2(\mu +
  2\pi^2X_{2,t}^2)^2.\)

\item[Others] Many other models, for example including covolatility or
  leverage, or combining any of the above features, can be described
  similarly. For simplicity, we assume in the following that the times
  \(t_i\) are deterministic and uniform; however, models with uneven
  or random times that are suitably dense and predictable can be
  addressed in a similar fashion.
\end{description}

To concisely describe these examples, and others, we will state a set
of assumptions on the observations \(Y_i,\) mean and variance
functions \(\mu\) and \(\sigma^2,\) parameters \(\theta,\) covariates
\(X_t,\) and estimates \(\widehat X_i.\) It will be possible to show
that the above models all lie within our assumptions, and we may thus
work within these assumptions with some generality.

To begin, we define some notation. Let \(\norm{\,\cdot\,}\) denote any
finite-dimensional vector norm; write \(a = O(b)\) if
\(\norm{a} \le C\norm{b},\) for some universal constant \(C;\) and
write \(a = O_p(b)\) if for each \(\varepsilon > 0,\) the random
variables \(a\) and \(b\) satisfy
\(\P(\norm{a} > C_\varepsilon \norm{b}) \le \varepsilon,\) for
universal constants \(C_\varepsilon.\)

We stress here that the implied constants \(C\) and \(C_\varepsilon\)
are universal; in statements such as \(a = O(1),\) we require the
supremum \(\sup\, \norm{a}\) over all such \(a\) to be bounded.  Given
a function \(f : X \to \R,\) we also define the supremum norm
\(\norm{f}_\infty \coloneqq \tsup_{x \in X} \abs{f(x)}.\)

Our assumptions are then as follows.

\begin{assumption}
  \label{ass:mod}
  Let \((\Omega, \mathcal F, (\mathcal F_t)_{t \in [0,1]}, \P)\) be a
  filtered probability space, with adapted unobserved mean, variance
  and covariate processes \(\mu_t \in \R,\) \(\sigma_t^2 \ge 0,\) and
  \(X_t \in \R^q,\) respectively. For \(0 \le t \le t+h \le 1,\)
  letting \(W_t\) denote either of the processes \(\mu_t\) or \(X_t,\)
  we have
\begin{equation}
\label{eq:wsm}
\begin{aligned}
  W_t &= O(1),\\
  \E[W_{t+h} - W_t \mid \mathcal F_t] &= O(h),\\
  \E[\norm{W_{t+h} - W_t}^2 \mid \mathcal F_t] &= O(h).
\end{aligned}
\end{equation}

For \(i = 0, \dots, n-1,\) we have \(\mathcal F_{t_{i+1}}\)-measurable
estimates \(\widehat X_i\) of \(X_{t_i},\) satisfying
\begin{equation}
\label{eq:xha}
\begin{aligned}
  \E[\norm{\widehat X_{i} - X_{t_i}}^2 \mid \mathcal
  F_{t_i}] &= O(n^{-1}),\\
  \E[\norm{\widehat X_{i} - X_{t_i}}^4 \mid \mathcal
  F_{t_i}] &= O(n^{-1}).
\end{aligned}
\end{equation}
We also have \(\mathcal F_{t_{i+1}}\)-measurable observations \(Y_i,\)
satisfying
\begin{equation}
\label{eq:ymb}
\begin{aligned}
  \E[Y_i \mid \mathcal F_{t_i}] &= \mu_{t_i} + O(n^{-1/2}),\\
  \Var[Y_i \mid \mathcal F_{t_i}] &= \sigma_{t_i}^2 + O(n^{-1/4}),\\
  \E[\abs{Y_i}^{4+\varepsilon} \mid \mathcal F_{t_i}] &= O(1),
\end{aligned}
\end{equation}
for a constant \(\varepsilon > 0.\) 

Under the null hypothesis \(H_0,\) we suppose our observations \(Y_i\)
are described by a parametric model,
\[\mu_t = \mu(\theta_0, t, X_t),\qquad \sigma_t^2 = \sigma^2(\theta_0,t,
X_t),\]
for known functions \(\mu,\,\sigma^2 : \Theta \times [0,1] \times \R^q
\to \R,\) and an unknown parameter \(\theta_0 \in \Theta.\) We suppose
that \(\Theta \subseteq \R^p\) is closed, and \(\sigma^2\) is
positive. We also suppose the functions \(\mu\) and \(\sigma^2\) are
locally Lipschitz in \(\theta,\) continuously differentiable in \(t,\)
and twice continuously differentiable in \(X.\) Finally, we suppose we
have a good estimate \(\widehat \theta\) of \(\theta_0,\) satisfying
\[\widehat \theta - \theta_0 = O_p(n^{-1/2}).\]
Under the alternative hypothesis \(H_1,\) we instead allow
\(\mu_t,\,\sigma_t\) unrestricted, and require only that \(\widehat
\theta = O_p(1).\)
\end{assumption}

To ensure the examples given above lie within \autoref{ass:mod}, we
must require that the parameter space \(\Theta \subseteq \R^p\) be
closed, and the model function \(\mu\) be locally Lipschitz in
\(\theta,\) continuously differentiable in \(t,\) and twice
continuously differentiable in \(X_t.\) These conditions should all be
satisfied for most common models.

We must further require the semimartingales \(X_t\) to be bounded, and
have bounded characteristics. In general, this assumption may not hold
directly; however, we can assume it without loss of generality using
standard localisation arguments.

In \applink, we then check that the above examples satisfy our
conditions on the processes \(\mu_t,\) \(\sigma_t,\) and \(X_t;\)
estimates \(\widehat X_i;\) and observations \(Y_i.\) Most of these
conditions follow from standard results on stochastic processes; where
necessary, higher-moment bounds can be proved using our
\autoref{lem:was} below.

To satisfy \autoref{ass:mod}, it remains to choose an estimate
\(\widehat \theta\) of \(\theta_0,\) having error \(O_p(n^{-1/2})\)
under \(H_0,\) and being \(O_p(1)\) under \(H_1.\) While our results
are agnostic as to the choice of \(\widehat \theta,\) a simple choice
is given by the least-squares estimate
\begin{equation}
\label{eq:ls}
\widehat \theta \coloneqq \arg \tmin_{\theta \in \Theta}
\tsum_{i=0}^{n-1} (Y_i - \mu(\theta, t_{i}, \widehat X_i))^2,
\end{equation}
which can be found by numerical optimisation.  Under standard
regularity assumptions for nonlinear regression, this estimate
\(\widehat \theta\) can be shown to satisfy our conditions, arguing
for example as in Section~5 of \citet{vetter_model_2012}.

Finally, we note that in the microstructure noise and stochastic
volatility models, we need to make \(n^2+1\) observations of the
underlying process \(X_t\) to construct the \(n\) estimates \(Y_i.\)
We may thus expect to achieve the square-root of any convergence rates
given below; such behaviour, however, is common to all approaches to
these problems in the literature.

We have thus shown that many different semimartingale goodness-of-fit
problems can be described by our \autoref{ass:mod}. Next, we will
describe our solutions to these problems.

\section{Wavelet detection tests}
\label{sec:tes}

To state our tests for the problems given by \autoref{ass:mod}, we
first consider the signal function
\[S_t(\theta) \coloneqq (\mu_t - \mu(\theta, t, X_t))/\sigma(\theta,t,
  X_t).\] This function measures the distance of the model mean
\(\mu\) from the true mean \(\mu_t,\) weighted by the model variance
\(\sigma^2.\) Under \(H_0,\) we have
\[S_t(\widehat \theta) \approx S_t(\theta_0) = 0,\] while under
\(H_1,\) we can in general expect \(\abs{S_t(\widehat \theta)}\) to be
large. We may thus reject \(H_0\) whenever an estimate of
\(S_t(\widehat \theta)\) is significantly different from zero.

To estimate the signal \(S_t(\theta),\) we will use wavelet
methods. Let \(\varphi\) and \(\psi\) be the Haar scaling function and
wavelet,
\[\varphi \coloneqq 1_{[0, 1)},\qquad
\psi \coloneqq 1_{[0, 1/2)}-1_{[1/2,1)},\]
and for \(j = 0, 1, \dots,\) \(k = 0, \dots, 2^j-1,\) define the
Haar basis functions
\[\varphi_{j,k}(t) \coloneqq 2^{j/2}\varphi(2^jt - k), \qquad
\psi_{j,k}(t) \coloneqq 2^{j/2}\psi(2^jt - k).\]
We can then describe \(S_t(\theta)\) in terms of its
scaling and wavelet coefficients
\[\alpha_{j,k}(\theta) \coloneqq \tint_0^1
\varphi_{j,k}(t)S_t(\theta)\,dt, \qquad \beta_{j,k}(\theta) \coloneqq \tint_0^1 \psi_{j,k}(t)S_t(\theta)\,dt.\]

To estimate these coefficients, we first pick a resolution level \(J
\in \N_0,\) so that \(2^J\) is of order \(n^{1/2}.\) We then estimate
the scaling coefficients \(\alpha_{J,k}(\theta)\) by
\[\widehat \alpha_{J,k}(\theta) \coloneqq n^{-1} \tsum_{i=0}^{n-1}
\varphi_{J,k}(t_{i})Z_i(\theta),\] where the normalised observations
\[Z_i(\theta) \coloneqq (Y_i - \mu(\theta, t_{i}, \widehat
X_i))/\sigma(\theta, t_{i}, \widehat X_i).\]
We note that for fixed \(\theta,\) these estimates can be computed in
linear time, as each observation \(Y_i\) contributes to only one
coefficient \(\widehat \alpha_{J,k}(\theta).\)

To estimate the coefficients \(\alpha_{0,0}(\theta)\) and
\(\beta_{j,k}(\theta),\) \(0 \le j < J,\) we then perform a fast
wavelet transform, obtaining estimates
\[\widehat 
\alpha_{0,0}(\theta) \coloneqq \tsum_l\widehat
\alpha_{J,l}(\theta)\tint_0^1\varphi_{J,l}\varphi_{0,0}, \qquad
\widehat \beta_{j,k}(\theta) \coloneqq \tsum_l\widehat
\alpha_{J,l}(\theta)\tint_0^1\varphi_{J,l}\psi_{j,k}.\]
We note that efficient implementations of this transformation, running
in linear time, are widely available.

To test our hypotheses, we will take the maximum size of these
estimated coefficients, producing test statistics
\[\widehat T(\theta) \coloneqq \tmax_{0 \le j<J,k}\, \abs{\widehat
  \alpha_{0,0}(\theta)},\,\abs{\widehat \beta_{j,k}(\theta)}.\]
We will show that under \(H_0,\) \(\widehat T(\widehat \theta)\) is
asymptotically Gumbel distributed, while under \(H_1,\) \(\widehat
T(\widehat \theta)\) will tend to be greater.

\begin{theorem}Let \autoref{ass:mod} hold.
  \label{thm:lim}
  \begin{enumerate}
  \item \label{it:h0}
    Under \(H_0,\)
    \[a^{-1}_{2^J}(n^{1/2}\widehat T(\widehat \theta) - b_{2^J})
    \overset{d}\to G\] uniformly, where the constants
    \begin{align*}
      a_{m} &\coloneqq (2 \log(m))^{-1/2},\\
      b_{m} &\coloneqq a_{m}^{-1} - \tfrac12 a_{m}\log (\pi \log(m)),
    \end{align*} and \(G\) denotes the standard Gumbel
    distribution.
  \item \label{it:h1} Under \(H_1,\)
    \[\widehat T(\widehat \theta) - T(\widehat \theta) =
    O_p(n^{-1/2}\log(n)^{1/2})\] uniformly, where
    \[T(\theta) \coloneqq \tmax_{0\le j<J,k}\,
    \abs{\alpha_{0,0}(\theta)},\,\abs{\beta_{j,k}(\theta)}.\]
  \end{enumerate}
\end{theorem}

We thus obtain that under \(H_0,\) \(\widehat T(\widehat \theta)\)
concentrates around zero at a rate \(n^{-1/2}\log(n)^{1/2}.\) Under
\(H_1,\) it concentrates at the same rate around the quantity
\(T(\widehat \theta),\) which measures the size of the signal
\(S_t(\widehat \theta).\) We can use this result to construct tests of
our hypotheses, and prove bounds on their performance; we first note
that for some of our bounds, we will require the following assumption.

\begin{assumption}
  \label{ass:ito}
  The processes \(\mu_t\) and \(X_t\) are
  It\=o semimartingales,
  \begin{align*}
    \mu_t &= \tint_0^t(b_s^\mu\,ds + (c_s^\mu)^T\,dB_s +
    \tint_\R f_s^{\mu}(x)\,\lambda(dx, ds)),\\
    X_{i,t} &= \tint_0^t(b_{i,s}^X\,ds + (c_{i,s}^X)^T\,dB_s +
    \tint_\R f_{i,s}^X(x)\,\lambda(dx, ds)),
  \end{align*}
  for a Brownian motion \(B_s \in \R^{q+1},\) independent Poisson
  random measure \(\lambda(dx,ds)\) having compensator \(dx\,ds,\)
  predictable processes \(b_s^\mu,\) \(b_{i,s}^X,\) \(c_s^\mu,\) \(c_{i,s}^X =
  O(1),\) and predictable functions \(f_s^\mu(x),\) \(f_{i,s}^X(x)\)
  satisfying \(\tint_\R 1\wedge \abs{f_s(x)}\,dx = O(1).\)
\end{assumption}

Under \autoref{ass:ito}, we thus have that \(\mu_t\) and \(X_t\) are
It\=o semimartingales, with bounded characteristics and
finite-variation jumps. This assumption holds for many common
financial models, if necessary after a suitable localisation
step. Using this condition, we are now ready to describe our tests,
and bound their performance.

\begin{theorem}
  \label{thm:ub}
  Let \autoref{ass:mod} hold, and for \(\alpha \in (0,1),\)
  define the Gumbel quantile
  \[q_{n,\alpha} \coloneqq -a_{2^J}\log(-\log(1 - \alpha)) +
  b_{2^J},\] and critical region
  \[C_{n,\alpha} \coloneqq \{n^{1/2}\widehat T(\widehat \theta) > q_{n,\alpha}\}.\]

  \begin{enumerate}
  \item \label{it:u0} Under \(H_0,\) we have \(\P[C_{n,\alpha}] \to \alpha\) uniformly.
  
  \item \label{it:u1} Under \(H_1,\) let \(M_n > 0\) be a fixed
    sequence with \(M_n \to \infty.\) If \(E_n\) is one of the events:
    \begin{enumerate}
    \item \label{it:ua} \(\{\norm{S(\widehat \theta)}_\infty \ge
      M_nn^{-1/4}\log(n)^{1/2}\},\) given also \autoref{ass:ito}; or
    \item \label{it:ub} \(\{\tmax_{0 \le j \le J,k}\,2^{j/2}\abs{\tint_{2^{-j}k}^{2^{-j}(k+1)} S_t(\widehat \theta)\,dt} \ge
      M_nn^{-1/2}\log(n)^{1/2}\};\)
    \end{enumerate}
    we have \(\P[E_n \setminus C_{n,\alpha}] \to 0\) uniformly.
  \end{enumerate}
\end{theorem}

We thus obtain that the test which rejects \(H_0\) on the event
\(C_{n,\alpha}\) is of asymptotic size \(\alpha,\) and under
\autoref{ass:ito}, can detect signals \(S_t(\widehat \theta)\) at the
rate \(n^{-1/4}\log(n)^{1/2}\) in supremum norm. We further have that,
even without \autoref{ass:ito}, our test can detect a signal whenever
the size of its mean over a dyadic interval is large.

In particular, if \(S_t(\widehat \theta) \propto e_t\) for some
non-zero deterministic process \(e_t,\) then \(e_t\) must have
non-zero integral over some dyadic interval \(2^{-j}[k, k+1).\) We
deduce that our test can detect signals in the fixed direction \(e_t\)
at the rate \(n^{-1/2}\log(n)^{1/2},\) without prior knowledge of
\(e_t.\)

We can further show that these detection rates are
near-optimal.

\begin{theorem}
  \label{thm:lb}
  Let \autoref{ass:mod} hold, and \(\delta_n >0\) be a fixed sequence
  with \(\delta_n \to 0.\) If \(E_n\) is one of the events:
  \begin{enumerate}
  \item \label{it:la} \(\{\norm{S(\widehat \theta)}_\infty \ge
    \delta_nn^{-1/4}\},\) given also \autoref{ass:ito}; or
  \item \label{it:lb} \(\{\tmax_k
    2^{j_n/2}\abs{\tint_{2^{-j_n}k}^{2^{-j_n}(k+1)} S_t(\widehat
      \theta)\,dt} \ge \delta_nn^{-1/2}\},\) for some \(j_n =
    0, \dots, J;\)
  \end{enumerate}
  then no sequence of critical regions \(C_n\) can satisfy
  \[\lim \tsup_n \P[C_n] < 1\]
  uniformly over \(H_0,\) and
  \[\P[E_n \setminus C_n] \to 0\]
  uniformly over \(H_1.\)
\end{theorem}

We thus conclude that our goodness-of-fit tests achieve the
near-optimal detection rate of \(n^{-1/4}\log(n)^{1/2}\) against
general nonparametric alternatives, in a wide variety of
semimartingale models. This result is already a significant
improvement over previous work; we note that similar methods do not
establish near-optimality for the procedures of
\citet{dette_test_2003}, for example, where the corresponding lower
bound would be \(n^{-1/3}.\)

Furthermore, we have shown that our method simultaneously provides
near-optimal detection rates against alternatives which are easier to
detect, including the case where the signal \(S_t(\widehat \theta)\)
lies in a fixed direction \(e_t.\) We may thus achieve good detection
rates in a fully nonparametric setting, without sacrificing
performance against fixed alternatives.

\section{Finite-sample tests}
\label{sec:emp}

We next consider the empirical performance of our tests. As convergence
to the Gumbel distribution can be quite slow, in the following, we
will consider a bootstrap version of our tests, which will be more
accurate in finite samples.

The general procedure is as follows. First, we estimate the parameters
\(\theta\) from the data, using some estimate \(\widehat \theta.\)
Next, we simulate many sets of observations \(Y_i^{(j)}\) from the
null hypothesis, with parameters chosen by \(\widehat \theta.\) Any
components of the null hypothesis not described by \(\theta,\) such as
drift or jump processes, are set to zero.

For each set of simulated observations \(Y_i^{(j)},\) we then compute
a parameter estimate \(\widehat \theta^{(j)},\) and statistic
\(\widehat T^{(j)}(\widehat \theta^{(j)}).\) Finally, we reject the
null hypothesis if the original statistic
\(\widehat T(\widehat \theta)\) is larger than the
\((1-\alpha)\)-quantile of the simulated statistics
\(\widehat T^{(j)}(\widehat \theta^{(j)}).\)

We now perform some simple Monte Carlo experiments on these tests.  We
will compare our tests to those of \citet{dette_test_2003},
\citet{dette_estimation_2006} and \citet{dette_testing_2008}, using
the same methodology as \citeauthor{dette_testing_2008}.  As in that
paper, we will generate Monte Carlo observations in the local
volatility setting \eqref{eq:lv}. We will then use our tests to
evaluate the goodness-of-fit of various parametric models for the
volatility.

In each case, we consider receiving \(n=100,\) 200 or 500
observations, and constructing confidence tests at the
\(\alpha = 5\%\) or 10\% level.  We then generate 1,000 realisations of
simulated data, compare our statistic against 1,000 bootstrap samples
in each realisation, and report the proportion of runs in which the
null hypothesis is rejected.

In our tests, we set the resolution level
\(J \coloneqq \lfloor \log_2(n) / 2 \rfloor,\) and use the
least-squares parameter estimates \(\widehat \theta\) given by
\eqref{eq:ls}. As the models we consider will be linear in the
parameters \(\theta,\) we will be able to compute these estimates in
closed form, as linear regressions.

\autoref{tab:results} then gives the observed rejection probabilities
of our tests in two models: a constant volatility model, where
\(\mu(x, t, \theta) = \theta;\) and a proportional volatility model,
where \(\mu(x, t, \theta) = \theta x^2.\) In each case, we give
results for our tests under a variety of null and alternative
hypotheses.

We note the hypotheses tested are the same as in Tables 1--4 of
\citet{dette_testing_2008}, as well as Table 3 of
\citet{dette_test_2003}, and Tables 3.1 and 3.4 of
\citet{dette_estimation_2006}. We may thus directly compare the
performance of our tests to those given in previous work.

We find that in both models, our tests have good coverage under the
null hypothesis, and reliably reject under the alternative
hypothesis. The power of our tests is competitive with previous work
under the constant volatility model, and generally improves upon
previous work under the proportional volatility model.

We conclude that our tests not only achieve good theoretical detection
rates, but also provide strong finite-sample performance. They may
thus be recommended for many different goodness-of-fit problems,
whether previously discussed in the literature, or newly described by
our more general assumptions.

\input{sdgft-tables.tex}

\section{Proofs}
\label{sec:pro}

We now give proofs of our results. In \autoref{sec:pst} we will state
some technical results, in \autoref{sec:pmn} give our main proofs, and
in \applink\ prove our technical results.

\subsection{Technical results}
\label{sec:pst}

We first state the technical results we will require.  Our main
technical result will be a central limit theorem for martingale
difference sequences, bounding the exponential moments of the distance
from Gaussian.

\begin{lemma}
  \label{lem:was}
  Let \((\Omega, \mathcal F, (\mathcal F_j)_{j=0}^n, \P)\) be a
  filtered probability space, and let \(X_i,\) \(i = 0, \dots, n-1,\)
  be \(\mathcal F_{i+1}\)-measurable real random variables.  Suppose
  that for some \(\kappa \ge 1,\)
  \begin{align*}
    \E[X_{i} \mid \mathcal F_{i}] &= 0,\\
    \tsum_{i=0}^{n-1} \E[\abs{X_{i}}^{4\kappa} \mid \mathcal F_{i}] &= O(n^{1-2\kappa}).
  \end{align*}

  \begin{enumerate}
  \item \label{it:gd} If also
    \[\E[\abs{\tsum_{i=0}^{n-1} \E[X_i^2 \mid \F_i] -
        1}^{2\kappa} \mid \F_0] = O(n^{-\kappa}),\] 
    then on a suitably-extended probability space, we
    have real random variables \(\xi,\) \(\eta\) and \(M,\)
    independent of \(\mathcal F\) given \(\mathcal F_n,\) such that
    \[\tsum_{i=0}^{n-1} X_i = \xi + \eta;\]
    \(\xi\) is standard Gaussian given \(\mathcal F_0;\) we have
    \[\E[\abs{\eta}^{4\kappa} \mid \F_0] = O(n^{-\kappa});\]
    for \(u \in \R,\)
    \[\E[\exp(u\eta - \tfrac12 u^2M) \mid \mathcal F_0] \le 1;\]
    and \(M \ge 0\) satisfies
    \begin{equation}
      \label{eq:ms}
      \E[M^{2\kappa} \mid \mathcal F_0] = O(n^{-\kappa}).
    \end{equation}

  \item \label{it:em} For random variables \(c_i = O(1),\) let
    \(\upsilon_c \coloneqq \tsum_{i=0}^{n-1}c_iX_i.\) Then on a
    suitably-extended probability space, we have a constant
    \(A = O(1)\) and real random variable \(M,\) independent of
    \(\mathcal F\) given \(\mathcal F_n,\) such that
    \[\tsup_c \E[\abs{\upsilon_c}^{4\kappa} \mid \F_0] = O(1);\]
    for \(u \in \R,\)
    \[\tsup_{c}\, \E[\exp(u\upsilon_c - \tfrac12 u^2(A+M))\mid \mathcal F_0]
      \le 1;\] and \(M \ge 0\) satisfies~\eqref{eq:ms}.
  \end{enumerate}
\end{lemma}

We will also need the following result on combining exponential moment
bounds.

\begin{lemma}
  \label{lem:tec}
  Let \((\Omega, \mathcal F, \P)\) be a probability space, with real
  random variables \((X_i)_{i=0}^{n-1}\) and \(M.\) Suppose that for
  \(u \in \R,\)
  \[\E[\exp(u X_i - \tfrac12 u^2 M)] = O(1),\]
  and \(M = O_p(r_n)\) for some rate \(r_n > 0.\) Then
  \[\tmax_i\, \abs{X_i} = O_p(r_n^{1/2}\log(n)^{1/2}).\]
\end{lemma}

Our next technical result will bound the moments of our observations
\(Y_i,\) and their normalisations \(Z_i(\theta).\) The result will be
stated using the H\"older spaces \(C^s,\) defined as follows.  Given a
function \(f : X \to \R,\) for suitable \(X \subseteq \R^d,\) we
define the 1-H\"older norm
\[\norm{f}_{C^1} \coloneqq \norm{f}_\infty \vee \tsup_{x,y \in X} \abs{f(x) -
  f(y)}/\norm{x-y},\]
and the \(2\)-H\"older norm
\[\norm{f}_{C^2} \coloneqq \begin{cases} 
  \norm{f}_{\infty} \vee \tmax_{i=1}^d \norm{(\nabla f)_i}_{C^1}, &\text{\(f\) is differentiable,}\\
  \infty, &\text{otherwise}.
\end{cases}\]
We also say \(f\) is \(C^s\) if \(\norm{f}_{C^s} < \infty.\)

\begin{lemma}
  \label{lem:obs}
  Under \(H_0\) or \(H_1,\) suppose the \(\widehat X_i = O(1),\) and
  \(\Theta\) is bounded.
  \begin{enumerate}
  \item \label{it:yl} For fixed \(i\) and \(Y_i,\) the variables
    \(Z_i(\theta)\) are \(C^1\) functions of \(\theta\) and \(\widehat
    X_i,\) with H\"older norm \(O(1 + \abs{Y_i}).\)
  \item \label{it:sl} The variables \(S_t(\theta)\) are \(C^1\)
    functions of \(\theta,\) \(t,\) \(\mu_t\) and \(X_t,\) and for
    fixed \(\theta\) and \(t,\) also \(C^2\) functions of \(\mu_t\)
    and \(X_t,\) both with H\"older norm \(O(1).\)
  \item \label{it:zm} For \(\theta \in \Theta,\) we have
    \begin{align*}
      \E[Z_i(\theta) \mid \mathcal F_{t_i}]
      &= S_{t_i}(\theta) + O(n^{-1/2}),\\
      \E[\abs{Z_i(\theta)}^{4+\varepsilon} \mid \mathcal F_{t_i}]
      &= O(1),\\
      \intertext{and under \(H_0,\) also}
      \E[Z_i(\theta_0)^2 \mid \mathcal F_{t_i}]
      &= 1 + O(n^{-1/4}).
    \end{align*}
  \item \label{it:ys}
    Define times
    \begin{equation}
      \label{eq:sk}
      s_k \coloneqq \lceil n2^{-J}k \rceil / n, \qquad k = 0, \dots,
      2^J.
    \end{equation}
    Then
    \[\tmax_k n^{-1/2} \tsum_{i=ns_k}^{ns_{k+1}-1} Y_i^2 = O_p(1).\]
\end{enumerate}
\end{lemma}

Finally, we will need a result controlling the behaviour of the
processes \(S_t(\theta)\) under \autoref{ass:ito}.

\begin{lemma}
  \label{lem:smt}
  Under \(H_1,\) suppose \(\Theta\) is bounded, let \(\underline
  \Theta_n \subseteq \Theta\) be a sequence of finite sets, of size
  \(O(n^\kappa)\) for some \(\kappa \ge 0,\) and let \(\delta_n = O(n^{-1/2}).\)
  Given \autoref{ass:ito}, we have
  \[S_t(\theta) = \widetilde S_t(\theta) + \overline S_t(\theta),\]
  where the processes \(\widetilde S_t(\theta)\) and \(\overline S_t(\theta)\) are as
  follows.
  \begin{enumerate}
  \item \label{it:hs} We have
    \[\tsup_{\theta \in\underline \Theta_n,\,\abs{s-t} \le \delta_n}
    \abs{\widetilde S_s(\theta) - \widetilde S_t(\theta)} =
    O_p(n^{-1/4}\log(n)^{1/2}).\]
  \item \label{it:sj} In \(L^2([0,1]),\) let \(P_Jf\) denote the
    orthogonal projection of \(f\) onto the subspace spanned by the
    scaling functions \(\varphi_{J,k},\) and define the remainder
    \(R_Jf \coloneqq f - P_Jf.\) Then
    \[\tsup_{\theta \in \underline \Theta_n}
    \norm{R_J\widetilde S(\theta)}_\infty = O_p(n^{-1/4}\log(n)^{1/2}).\]
  \item \label{it:lj} We have a random variable \(N \in \N,\) and
    random times \(0 = \tau_0 < \dots < \tau_N = 1,\) such that the
    processes \(\overline S_t(\theta),\) \(\theta \in \underline \Theta_n,\) are
    constant on intervals \([\tau_i, \tau_{i+1}),\)
    \([\tau_{N-1},\tau_N],\) and
    \[\P[\tmin_i (\tau_{i+1} - \tau_i) < \delta_n] \to 0.\]
  \end{enumerate}
\end{lemma}

\subsection{Main proofs}
\label{sec:pmn}

We may now proceed with our main proofs. We first prove
\autoref{thm:lim}, beginning with a lemma controlling the variance of
our estimated scaling coefficients \(\widehat \alpha_{J,k}(\theta).\)

\begin{lemma}
  \label{lem:alp}
  For \(k = 0, \dots, 2^J-1,\) \(\theta \in \Theta,\) define
  scaling-coefficient variance terms
  \[\widetilde \alpha_{J,k}(\theta) \coloneqq n^{-1}\tsum_{i = 0}^{n-1}
    \varphi_{J,k}(t_{i})(Z_i(\theta) - \E[Z_i(\theta) \mid \mathcal F_{t_{i}}]).\]

  \begin{enumerate}
  \item \label{it:gd2} Under \(H_0,\) suppose the \(\widehat X_i =
    O(1).\) Then on a suitably-extended probability space, we have a
    filtration \((\mathcal G_k)_{k=0}^{2^J},\) and \(\mathcal
    G_{k+1}\)-measurable real random variables \(\xi_{k},\)
    \(\eta_{k},\) \(M_k,\) such that
    \[n^{1/2}\widetilde \alpha_{J,k}(\theta_0) = \xi_{k} + \eta_{k};\]
    the variables \(\xi_{k}\) are standard Gaussian given \(\mathcal
    G_{k};\)
    \[\E[\exp(u\eta_{k} - \tfrac12 u^2 M_k) \mid \mathcal G_{k}]
    \le 1;\]
    and the variables \(M_k \ge 0\) satisfy
    \begin{equation}
      \label{eq:mk}
      \E[M_k^{2 + \varepsilon/2} \mid \mathcal G_{k}] =
      O(n^{-(1/2+\varepsilon/8)}).
    \end{equation}

  \item \label{it:em2} Under \(H_1,\) suppose \(\Theta\) is bounded,
    and the \(\widehat X_i = X_{t_{i}}.\) We then have constants
    \(A_k = O(1),\) and on a suitably-extended probability space, a
    filtration \((\mathcal G_k)_{k=0}^{2^J}\) and real random
    variables \(M_k,\) such that
    \[\tsup_{\theta \in \Theta}\E[\exp(un^{1/2}\widetilde \alpha_{J,k}(\theta) - \tfrac12 u^2 (A_k +
    M_k)) \mid \mathcal G_{k}] \le 1;\]
    the variables \(M_k \ge 0\) satisfy~\eqref{eq:mk}; and the
    \(\widetilde \alpha_{J,k}(\theta)\) and \(M_k\) are \(\mathcal
    G_{k+1}\)-measurable.
  \end{enumerate}
\end{lemma}

\begin{proof}
  We first prove part~\eqref{it:gd2}, and argue by induction on \(k.\)
  Let \(\mathcal G_0 = \mathcal F_0,\) and suppose that for \(i = 0,
  \dots, k-1\) we have constructed, on an extended probability space,
  \(\sigma\)-algebras \(\mathcal G_{i+1},\) and random variables
  \(\xi_i,\) \(\eta_i,\) \(M_i\) satisfying our conditions.  We
  suppose also that \(\mathcal G_{k}\) has been chosen to be
  independent of \(\mathcal F\) given \(\mathcal F_{s_{k}},\) where
  the times \(s_k\) are given by~\eqref{eq:sk}; we note this condition
  is trivially satisfied for \(\mathcal G_0.\)

  We can then write
  \[n^{1/2}\widetilde \alpha_{J,k}(\theta_0) = \tsum_{i =
    ns_k}^{ns_{k+1}-1} \zeta_{i},\]
  where the \(m \coloneqq n(s_{k+1}-s_k)\) summands
  \[\zeta_{i} \coloneqq n^{-1/2}2^{J/2}(Z_i(\theta_0) -
  \E[Z_i(\theta_0) \mid \mathcal F_{t_{i}}]).\]
  To compute the moments of the \(\zeta_i,\) we may apply
  \autoref{lem:obs}\eqref{it:zm}, noting that since we are only
  interested in \(\theta = \theta_0,\) we may assume \(\Theta\) is
  bounded. We thus have
  \begin{align*}
    \E[\zeta_{i} \mid \mathcal F_{t_{i}}, \mathcal G_{k}]
    &= 0,\\
    \tsum_{i=ns_k}^{ns_{k+1}-1} \E[\zeta_{i}^2 \mid
    \mathcal F_{t_{i}}, \mathcal G_{k}]
    &= 1 + O(m^{-1/2}),\\
    \tsum_{i=ns_k}^{ns_{k+1}-1} \E[\abs{\zeta_{i}}^{4+\varepsilon}
    \mid \mathcal F_{t_{i}}, \mathcal G_{k}]
    &= O(m^{-(1 + \varepsilon/2)}),
  \end{align*}
  using also that the \(\zeta_i\) are independent of \(\mathcal
  G_{k}\) given \(\mathcal F_{t_{i}}.\)

  We may therefore apply \autoref{lem:was}\eqref{it:gd} to the
  variables \(n^{1/2}\widetilde \alpha_{J,k}(\theta_0).\) On a
  further-extended probability space, we obtain random variables
  \(\xi_k,\) \(\eta_k,\) \(M_k\) satisfying the conditions of
  part~\eqref{it:gd2}, independent of \(\mathcal F\) given \(\mathcal
  G_{k}\) and \(\mathcal F_{s_{k+1}}.\) Defining \(\mathcal G_{k+1}\)
  to be the \(\sigma\)-algebra generated by \(\mathcal G_{k},\)
  \(\mathcal F_{s_{k+1}},\) \(\xi_k,\) \(\eta_k\) and \(M_k,\) we
  deduce that \(\mathcal G_{k+1}\) satisfies the conditions of our
  inductive hypothesis.  By induction, we conclude that
  part~\eqref{it:gd2} of our result holds.

  To prove part~\eqref{it:em2}, we argue similarly, noting that the
  random variables
  \[n^{1/2}\widetilde \alpha_{J,k}(\theta) = \tsum_{i =
    ns_k}^{ns_{k+1}-1} c_i(\theta)\widetilde \zeta_{i},\]
  where the \(\mathcal F_{t_{i+1}}\)-measurable summands
  \[\widetilde \zeta_{i} \coloneqq n^{-1/2}2^{J/2}(Y_i - \E[Y_i \mid
  \mathcal F_{t_i}]),\]
  and the \(\mathcal F_{t_{i}}\)-measurable coefficients
  \[c_i(\theta) \coloneqq 1/\sigma(\theta,t_{i}, X_{t_{i}}).\]
  As the function \(\sigma\) is continuous and positive, and
  \(\theta\) and \(X_t\) are bounded, we have the variables
  \(c_i(\theta) = O(1).\) We may thus apply
  \autoref{lem:was}\eqref{it:em}, producing random variables \(A_k,\)
  \(M_k\) satisfying the conditions of part~\eqref{it:em2}. The result
  then follows as before.
\end{proof}

We now prove a lemma bounding the variance of our estimated scaling
and wavelet coefficients \(\widehat \alpha_{0,0}(\theta),\) \(\widehat
\beta_{j,k}(\theta).\)

\begin{lemma}
  \label{lem:var}
  Suppose the \(\widehat X_i = O(1),\) and for \(j = 0, \dots, J-1,\)
  \(k = 0, \dots, 2^j-1\) and \(\theta \in \Theta,\) define the
  wavelet-coefficient variance terms
  \[\widetilde \beta_{j,k}(\theta) \coloneqq n^{-1}\tsum_{i = 0}^{n-1}
  \psi_{j,k}(t_{i})(Z_i(\theta) - \E[Z_i(\theta) \mid \mathcal
  F_{t_{i}}]).\]
  Similarly define scaling-coefficient variance terms \(\widetilde
  \alpha_{0,0}(\theta)\) using \(\varphi_{0,0}.\)

  \begin{enumerate}
  \item \label{it:nv} Under \(H_0,\) suppose \(\widehat \theta -
    \theta_0 = O(n^{-1/2}).\) Then on a suitably-extended probability
    space, we have real random variables \(\widetilde \xi_{j,k},
    \widetilde \eta_{j,k},\widetilde \upsilon_{j,k}\) such that
    \begin{align*}
      n^{1/2}\widetilde \alpha_{0,0}(\widehat \theta)
      &= \widetilde \xi_{-1,0} +
        \widetilde \eta_{-1,0} + \widetilde \upsilon_{-1,0};\\
      n^{1/2}\widetilde \beta_{j,k}(\widehat \theta)
      &= \widetilde \xi_{j,k} + \widetilde \eta_{j,k} + \widetilde \upsilon_{j,k};
    \end{align*}
    the \(\widetilde \xi_{j,k}\) are independent standard Gaussian;
    and for some \(\varepsilon' > 0,\)
    \[\tmax_{j,k}\, \abs{\widetilde \eta_{j,k}} =
    O_p(n^{-\varepsilon'}), \qquad \tmax_{j,k}\,2^{j/2}\abs{\widetilde
      \upsilon_{j,k}} = O_p(1).\]

  \item \label{it:av}
    Under \(H_1,\) suppose \(\Theta\) is bounded. Then
    \[\tsup_{j,k,\theta \in \Theta}\, \abs{\widetilde
      \alpha_{0,0}(\theta)},\,\abs{\widetilde \beta_{j,k}(\theta)} =
    O_p(n^{-1/2}\log(n)^{1/2}).\]
  \end{enumerate}  
\end{lemma}

\begin{proof}
  We will consider the wavelet-coefficient variance terms \(\widetilde
  \beta_{j,k}(\theta);\) we note we may include scaling-coefficient
  variance terms \(\widetilde \alpha_{0,0}(\theta)\) similarly.  To
  prove part~\eqref{it:nv}, we then apply
  \autoref{lem:alp}\eqref{it:gd2}.  We obtain a filtration \(\mathcal
  G_l,\) and variables \(M_l,\) \(\xi_{l}\) and \(\eta_{l}\) as in the
  statement of the lemma.  Since
  \[\widetilde \beta_{j,k}(\theta) =
  \tsum_l b_{j,k,l} \widetilde \alpha_{J,l}(\theta),\]
  where the coefficients
  \[b_{j,k,l} \coloneqq \tint_0^1 \psi_{j,k}\varphi_{J,l},\]
  we have
  \[n^{1/2}\widetilde \beta_{j,k}(\widehat \theta) = \widetilde \xi_{j,k} +
  \widetilde \eta_{j,k} + \widetilde \upsilon_{j,k},\] for terms
  \[\widetilde \xi_{j,k} \coloneqq \tsum_{l}  b_{j,k,l} \xi_{l},\qquad
  \widetilde \eta_{j,k} \coloneqq \tsum_{l} b_{j,k,l} \eta_{l},\]
  and
  \[\widetilde \upsilon_{j,k} \coloneqq n^{1/2}(\widetilde
  \beta_{j,k}(\widehat \theta) - \widetilde
  \beta_{j,k}(\theta_0)).\]

  We first describe the terms \(\widetilde \xi_{j,k}.\) Since the
  \(\xi_l\) are jointly centred Gaussian, so are the \(\widetilde
  \xi_{j,k}.\) Furthermore, we have
  \begin{align}
    \notag \Cov[\widetilde \xi_{j,k} \widetilde \xi_{j',k'}] &= 
    \tsum_l b_{j,k,l}b_{j',k',l}\\
    \notag &= \tint_0^1(\tsum_l b_{j,k,l}\varphi_{J,l})(\tsum_{l'} b_{j',k',l'}\varphi_{J,l'})\\
    \notag &= \tint_0^1
    \psi_{j,k}\psi_{j',k'}\\
    \label{eq:bs} &= 1_{(j,k)=(j',k')}.
  \end{align}
  We deduce that the \(\widetilde \xi_{j,k}\) are independent standard
  Gaussian.

  We next bound the \(\widetilde \eta_{j,k}.\) Setting
  \[M \coloneqq \tmax_l M_l,\] we have that
  \[\E[M^{2+\varepsilon/2}] \le \tsum_l
  \E[M_l^{2+\varepsilon/2}] = O(n^{-\varepsilon/8}),\]
  so \(M = O_p(n^{-\varepsilon'})\) for some \(\varepsilon' > 0.\)
  Using~\eqref{eq:bs}, we also have
  \begin{align*}
    \E[\exp(u\widetilde \eta_{j,k} - \tfrac12 u^2
    M)]
    &\le \E[\tprod_l\exp(ub_{j,k,l}\eta_{l} -
    \tfrac12 u^2 b_{j,k,l}^2M_l)]\\
    &\le 1.
  \end{align*}
  The desired result follows by applying \autoref{lem:tec}. 

  Finally, we control the \(\widetilde \upsilon_{j,k}.\) Since we are only
  interested in \(\theta = \theta_0,\,\widehat \theta,\) we may assume
  \(\Theta\) is bounded. For \(\theta,\theta' \in \Theta,\)
  \(\abs{\theta - \theta'} = O(n^{-1/2}),\) we then have
  \begin{align}
    \notag \mel\tsup_{j,k,\theta,\theta'} 2^{j/2}
    \abs{\widetilde \beta_{j,k}(\theta) - \widetilde
    \beta_{j,k}(\theta')}\\
    \notag &= \tmax_{j,k} O(n^{-3/2}2^{j/2})
      \tsum_{i=0}^{n-1} \abs{\psi_{j,k}(t_{i})}(1 + \abs{Y_i}),\\
    \intertext{using \autoref{lem:obs}\eqref{it:yl},}
    \notag &= O(n^{-1/2})(1 + \tmax_k n^{-1/2} \tsum_{i=ns_k}^{ns_{k+1}-1} \abs{Y_i})\\
    \notag &= O(n^{-1/2})(1 + (\tmax_k n^{-1/2}
             \tsum_{i=ns_k}^{ns_{k+1}-1} Y_i^2)^{1/2}),\\
    \intertext{by Cauchy-Schwarz,}
    \label{eq:bt} &= O_p(n^{-1/2}),
  \end{align}
  using \autoref{lem:obs}\eqref{it:ys}.  We deduce that
  \[\tsup_{j,k}2^{j/2}\abs{\widetilde \upsilon_{j,k}} = O_p(1).\]

  To prove part~\eqref{it:av}, we first claim we may assume the
  \(\widehat X_i = X_{t_{i}}.\) To prove the claim, we define terms
  \[Z_i'(\theta) \coloneqq (Y_i - \mu(\theta, t_{i},
  X_{t_{i}}))/\sigma(\theta, t_{i}, X_{t_{i}}),\]
  and
  \[\widetilde \beta_{j,k}'(\theta) \coloneqq n^{-1}\tsum_{i = 0}^{n-1}
  \psi_{j,k}(t_{i})(Z_i'(\theta) - \E[Z_i'(\theta) \mid \mathcal
  F_{t_{i}}]).\]
  We then have
  \begin{align*}
    \mel
    \tsup_{j,k,\theta \in \Theta} \abs{\widetilde
    \beta_{j,k}(\theta) - \widetilde \beta_{j,k}'(\theta)}\\
 &= O(n^{-1})\tmax_{j,k}\tsum_{i=0}^{n-1} \abs{\psi_{j,k}(t_{i})}(1 +
   \abs{Y_i})\norm{\widehat X_i - X_{t_{i}}},\\
    \intertext{using \autoref{lem:obs}\eqref{it:yl},}
 &= \bal O(n^{-1/2})(\tmax_{j,k}n^{-1}\tsum_{i=0}^{n-1} \psi_{j,k}^2(t_{i})(1 +
   Y_i^2))^{1/2}\\
 &\times (\tsum_{i=0}^{n-1}\norm{\widehat X_i -
   X_{t_{i}}}^2)^{1/2},\eal\\
    \intertext{by Cauchy-Schwarz,}
 &= O_p(n^{-1/2})(1 + \tmax_{k}n^{-1/2}\tsum_{i=ns_k}^{ns_{k+1}-1} Y_i^2)^{1/2},\\
    \intertext{since \(\E[\tsum_{i=0}^{n-1} \norm{\widehat X_i -
    X_{t_{i}}}^2] = O(1),\)}
 &= O_p(n^{-1/2}),
  \end{align*}
  using \autoref{lem:obs}\eqref{it:ys}.

  We may thus assume the \(\widehat X_i = X_{t_{i}},\) and so apply
  \autoref{lem:alp}\eqref{it:em2}. On an extended probability space,
  we obtain a filtration \(\mathcal G_l,\) constants \(A_l = O(1),\)
  and variables \(M_l\) as in the statement of the lemma. Setting
  \[M \coloneqq \tmax_l (A_l + M_l),\]
  we obtain that \(M = O_p(1),\) and
  \[\tsup_{\theta \in \Theta} \E[\exp(un^{1/2}\widetilde \beta_{j,k}(\theta) - \tfrac12 u^2
    M)] \le 1,\] arguing as in part~\eqref{it:gd2}.  Letting
  \(\underline \Theta_n\) denote a \(n^{-1/2}\)-net for
  \(\Theta \subset \R^p,\) of size \(O(n^{p/2}),\) we thus have
  \begin{align*}
    \tmax_{j,k,\theta \in \underline \Theta_n}
    \abs{\widetilde \beta_{j,k}(\theta)} &= O_p(n^{-1/2}\log(n^{p/2})^{1/2})\\
    &= O_p(n^{-1/2}\log(n)^{1/2}),
  \end{align*}
  using \autoref{lem:tec}.

  Next, for any \(\theta \in \Theta,\) we have a point
  \(\underline \theta \in \underline \Theta_n\) with \(\theta -
  \underline \theta = O(n^{-1/2}).\) Using~\eqref{eq:bt}, we deduce
  that
  \[\tsup_{j,k,\theta \in \Theta}
    \abs{\widetilde \beta_{j,k}(\theta) - \widetilde
    \beta_{j,k}(\underline \theta)} = O_p(n^{-1/2}).\]
  We conclude that
  \[\tsup_{j,k,\theta \in \Theta} \abs{\widetilde \beta_{j,k}(\theta)} =
  O_p(n^{-1/2}\log(n)^{1/2}).\qedhere\]
\end{proof}

Next, we prove a lemma bounding the bias of our estimated scaling and
wavelet coefficients \(\widehat \alpha_{0,0}(\theta),\) \(\widehat
\beta_{j,k}(\theta).\)

\begin{lemma}
  \label{lem:bia}
  Suppose the \(\widehat X_i = O(1),\) and for \(j = 0, \dots, J-1,\)
  \(k = 0, \dots, 2^J-1\) and \(\theta \in \Theta,\) define the
  wavelet-coefficient bias terms
  \[\overline \beta_{j,k}(\theta) \coloneqq n^{-1}\tsum_{i = 0}^{n-1}
  \psi_{j,k}(t_{i})\E[Z_i(\theta) \mid \mathcal F_{t_{i}}] -
  \beta_{j,k}(\theta),\]
  Similarly define scaling-coefficient bias terms \(\overline
  \alpha_{0,0}(\theta)\) using \(\varphi_{0,0}.\)

  \begin{enumerate}
  \item \label{it:nb} Under \(H_0,\) suppose \(\widehat \theta -
    \theta_0 = O(n^{-1/2}).\) Then
    \[\tmax_{j,k} \abs{\overline \alpha_{0,0}(\widehat \theta)},
    \,2^{j/2}\abs{\overline \beta_{j,k}(\widehat \theta)} =
    O_p(n^{-1/2}).\]
  \item \label{it:ab}
    Under \(H_1,\) suppose \(\Theta\) is bounded. Then
    \[\tsup_{j,k,\theta \in \Theta} \abs{\overline \alpha_{0,0}(\theta)},\,\abs{\overline \beta_{j,k}(\theta)} =
    O_p(n^{-1/2}).\]
  \end{enumerate}
\end{lemma}

\begin{proof}
  We will bound the wavelet-coefficient bias terms \(\overline
  \beta_{j,k}(\theta);\) we note we may include the
  scaling-coefficient bias terms \(\overline \alpha_{0,0}(\theta)\)
  similarly.  For \(t \in [0,1],\) define \(\underline t \coloneqq
  \lfloor n t \rfloor / n,\) and set
  \[\underline \beta_{j,k}(\theta) \coloneqq
  \tint_0^1 \psi_{j,k}(t)(S_{\underline t}(\theta) -
  S_t(\theta))\,dt.\]
  In each part~\eqref{it:nb} and~\eqref{it:ab}, we will show that
  \(\overline \beta_{j,k}(\theta)\) is close to \(\underline
  \beta_{j,k}(\theta),\) which is small.

  We note that in either part we may assume \(\Theta\) is bounded,
  since in part~\eqref{it:nb}, we are only interested in \(\theta =
  \theta_0,\,\widehat \theta.\) We then have
  \begin{align}
    \notag \abs{\overline \beta_{j,k}(\theta) - \underline
    \beta_{j,k}(\theta)} &\le \bal n^{-1}\tsum_{i=0}^{n-1}
    \abs{\psi_{j,k}(t_{i})}\abs{\E[Z_i(\theta) \mid
      \mathcal
      F_{t_{i}}] - S_{t_i}(\theta)}\\
    &+ \tint_0^1 \abs{\psi_{j,k}(t) - \psi_{j,k}(\underline
      t)}\abs{S_{\underline t}(\theta)}\,dt\eal\\
    \label{eq:ob} &= O(n^{-1/2}2^{-j/2}),
  \end{align}
  using \hyperref[lem:obs]{Lemmas}~\ref{lem:obs}\eqref{it:sl}
  and~\eqref{it:zm}.  It thus remains to bound the \(\underline
  \beta_{j,k}(\theta).\)

  To prove part~\eqref{it:nb}, we note that
  \[\underline \beta_{j,k}(\theta_0) = \tsum_{i=0}^{n-1} \zeta_{i,j,k},\]
  where the \(\mathcal F_{t_{i+1}}\)-measurable summands
  \[\zeta_{i,j,k} \coloneqq -\tint_{t_{i}}^{t_{i+1}} \psi_{j,k}(t)
  (S_{t}(\theta_0) - S_{t_i}(\theta_0))\,dt.\]
  Using \autoref{lem:obs}\eqref{it:sl} and Taylor's theorem, we also
  have that
  \begin{align*}S_{t}(\theta_0) - S_{t_i}(\theta_0)
    = \bal c_{i}(\mu_t -
  \mu_{t_{i}}) + d_{i}^T(X_t - X_{t_{i}})\\
  &+ O(\abs{\mu_t -
  \mu_{t_{i}}}^2 + \norm{X_t - X_{t_{i}}}^2 + n^{-1}),\eal
  \end{align*}
  for bounded \(\mathcal F_{t_i}\)-measurable random variables \(c_i
  \in \R,\) \(d_i \in \R^q.\)

  We deduce that
  \[\E[\zeta_{i,j,k} \mid \mathcal F_{t_{i}}]
  = O(n^{-2}2^{j/2}),\]
  and similarly
  \[ \Var[\zeta_{i,j,k} \mid \mathcal F_{t_{i}}] \le
  \E[\zeta_{i,j,k}^2 \mid \mathcal F_{t_{i}}] = O(n^{-3}2^{j}).\]
  Furthermore, for fixed \(j\) and \(k,\) we have that all but
  \(O(n2^{-j})\) of the \(\zeta_{i,j,k}\) are almost-surely zero.  We
  thus have
  \[\E[\underline \beta_{j,k}(\theta_0)^2] = O(n^{-2}).\]

  We deduce that
  \begin{align*}
    \E[\tmax_{j,k} \underline \beta_{j,k}(\theta_0)^2]
    &\le \tsum_{j,k}\E[\underline \beta_{j,k}(\theta_0)^2]\\
    &\le O(n^{-2}) \tsum_j 2^j\\
    &= O(n^{-3/2}),
  \end{align*}
  so \(\tmax_{j,k} \abs{\underline \beta_{j,k}(\theta_0)} =
  O_p(n^{-3/4}).\) We also have
  \[\underline \beta_{j,k}(\theta_0) - \underline
  \beta_{j,k}(\widehat \theta) = O(n^{-1/2}2^{-j/2}),\] using
  \autoref{lem:obs}\eqref{it:sl}. We conclude that
  \begin{align*}
    \tmax_{j,k} 2^{j/2}\abs{\overline
      \beta_{j,k}(\widehat \theta)} &\le \bal\tmax_{j,k}
    2^{j/2}\abs{\underline \beta_{j,k}(\theta_0)}\\
    &+ \tmax_{j,k} 2^{j/2}\abs{\underline
      \beta_{j,k}(\theta_0) - \underline \beta_{j,k}(\widehat \theta)}\\
    &+ \tmax_{j,k} 2^{j/2}\abs{\underline
      \beta_{j,k}(\widehat \theta) - \overline
      \beta_{j,k}(\widehat \theta)}\eal\\
    &= O_p(n^{-1/2}),
  \end{align*}
  using~\eqref{eq:ob}.

  To prove part~\eqref{it:ab}, using \autoref{lem:obs}\eqref{it:sl},
  we have
  \[S_{\underline t}(\theta) - S_{t}(\theta) = O(\abs{\mu_{\underline
      t} - \mu_t} + \norm{X_{\underline t} - X_t} + n^{-1}).\] We deduce that
  \begin{align*}
    \mel \tsup_{j,k,\theta \in \Theta}\abs{\underline
      \beta_{j,k}(\theta)}\\
    &= O(1) \tsup_{j,k} \tint_0^1 \abs{\psi_{j,k}(t)}
    (\abs{\mu_{\underline t} - \mu_t} + \norm{X_{\underline
        t} - X_t} + n^{-1})\,dt\\
    &= \bal O(1) (\tsup_{j,k} \tint_0^1 \psi^2_{j,k}(t)\,dt)^{1/2}\\
      &\times (\tint_0^1 (\abs{\mu_{\underline t} - \mu_t}^2 +
    \norm{X_{\underline t} - X_{t}}^2 + n^{-2})\,dt)^{1/2},\eal\\
    \intertext{by Cauchy-Schwarz,} &= O_p(n^{-1/2}),
  \end{align*}
  since \(\tint_0^1 \psi_{j,k}^2(t)\,dt = 1,\) and
  \[\E[\tint_0^1 (\abs{\mu_{\underline
      t} - \mu_t}^2 + \norm{X_{\underline t} - X_t}^2 + n^{-2})\,dt] =
    O(n^{-1}).\] Using~\eqref{eq:ob}, we conclude that
  \[\tsup_{j,k,\theta \in \Theta} \abs{\overline
    \beta_{j,k}(\theta)} = O_p(n^{-1/2}).\qedhere\]
\end{proof}

We can now prove our limit theorem for the statistic \(\widehat
T(\widehat \theta).\)

\begin{proof}[Proof of \autoref{thm:lim}]
  We first note that our estimated scaling and wavelet coefficients are
  equivalently given by
  \[\widehat \alpha_{0,0}(\theta) = n^{-1} \tsum_{i=0}^{n-1}
  \varphi_{0,0}(t)Z_i(\theta), \qquad \widehat \beta_{j,k}(\theta) =
  n^{-1} \tsum_{i=0}^{n-1} \psi_{j,k}(t)Z_i(\theta).\]
  We may thus make the variance-bias decomposition
  \begin{align*}
    \widehat \alpha_{0,0}(\theta) - \alpha_{0,0}(\theta) &= \widetilde
    \alpha_{0,0}(\theta) + \overline \alpha_{0,0}(\theta),\\
    \widehat \beta_{j,k}(\theta) - \beta_{j,k}(\theta) &= \widetilde
    \beta_{j,k}(\theta) + \overline \beta_{j,k}(\theta),
  \end{align*}
  where the terms \(\widetilde \alpha_{0,0},\) \(\overline
  \alpha_{0,0},\) \(\widetilde \beta_{j,k}\) and \(\overline
  \beta_{j,k}\) are defined by
  \hyperref[lem:var]{Lemmas}~\ref{lem:var} and~\ref{lem:bia}.  We will
  proceed to bound the distribution of \(\widehat T(\widehat \theta)\)
  using these lemmas.

  We begin by showing we may assume the estimated covariates
  \(\widehat X_i = O(1).\) We note that
  \[\E[\tmax_i \norm{\widehat X_i}^2] \le \E[\tsup_t \norm{X_t}^2]
  + \tsum_i \E[\norm{\widehat X_i - X_{t_{i}}}^2] = O(1),\]
  so \(\max_i \norm{\widehat X_i} = O_p(1).\) For a constant \(R >
  0,\) define the variables
  \[\widetilde X_i \coloneqq
  \begin{cases} \widehat X_i,
    &\norm{\widehat X_i} \le R,\\X_{t_{i}},
    &\text{otherwise}.
  \end{cases}\]
  Then as \(R \to \infty,\) the probability that the \(\widetilde
  X_i\) and \(\widehat X_i\) agree tends to one, uniformly in \(n.\)
  It thus suffices to prove our results replacing the \(\widehat X_i\)
  with the \(\widetilde X_i;\) equivalently, we may assume the
  \(\widehat X_i = O(1).\)

  We now prove part~\eqref{it:h0}. Since \(\widehat \theta - \theta_0
  = O_p(n^{-1/2}),\) we may similarly assume \(\widehat \theta -
  \theta_0 = O(n^{-1/2}).\) Let \(J_2 = \lfloor J / 2 \rfloor,\) and write
  \[\widehat T(\theta) = \max(\overline T(\theta),\widetilde T(\theta)),\]
  where the terms
  \begin{align*}
    \overline T(\theta) &\coloneqq \tmax_{0\le j< J_2, k}\, \abs{\widehat \alpha_{0,0}(\theta)},\,\abs{\widehat \beta_{j,k}(\theta)},\\
    \widetilde T(\theta) &\coloneqq \tmax_{J_2 \le j < J, k}
    \,\abs{\widehat \beta_{j,k}(\theta)}.
  \end{align*}
  Under \(H_0,\) using
  \hyperref[lem:var]{Lemmas}~\ref{lem:var}\eqref{it:nv}
  and~\ref{lem:bia}\eqref{it:nb}, we can then write
  \begin{align*}
    n^{1/2}\overline T(\widehat \theta) &= \tmax_{0\le j< J_2,k}\,
    \abs{\widetilde \xi_{j,k}} + O_p(1),\\
    n^{1/2}\widetilde T(\widehat \theta) &= \tmax_{J_2 \le j < J,k}\,
    \abs{\widetilde \xi_{j,k}} + O_p(n^{-\varepsilon'}),
  \end{align*}
  for some \(\varepsilon' > 0,\) and independent standard Gaussians
  \(\widetilde \xi_{j,k}.\)

  By standard Gumbel limits, we also have
  \begin{align*}
    a^{-1}_{2^{J_2}}(\tmax_{0\le j<J_2,k}\,\abs{\widetilde \xi_{j,k}} -
    b_{2^{J_2}}) &\overset{d}\to G,\\
    a^{-1}_{2^J}(\tmax_{J_2 \le j<J,k}\,\abs{\widetilde \xi_{j,k}} -
    b_{2^J}) &\overset{d}\to G;
  \end{align*}
  we note that in the second limit, we may use the constants
  \(a_{2^J}\) and \(b_{2^J},\) rather than \(a_{2^J-2^{J_2}}\) and
  \(b_{2^J-2^{J_2}},\) as the difference is negligible.  We deduce
  that
  \[\P[\widehat T(\widehat \theta) = \widetilde T(\widehat \theta)] \to 1,\]
  and so
  \begin{align*}
    a^{-1}_{2^J}(n^{1/2}\widehat T(\widehat \theta) - b_{2^J})
    &= a^{-1}_{2^J}(n^{1/2}\widetilde T(\widehat \theta) - b_{2^J}) + o_p(1)\\
    &= a^{-1}_{2^J}(\tmax_{J_2 \le j < J,k}
      \abs{\widetilde \xi_{j,k}} - b_{2^J}) + o_p(1)\\
    &\overset{d}\to G.
  \end{align*}

  Next, we prove part~\eqref{it:h1}. As before, since \(\widehat
  \theta = O_p(1),\) we may assume \(\widehat \theta = O(1),\) and
  hence that \(\Theta\) is bounded. Using
  \hyperref[lem:var]{Lemmas}~\ref{lem:var}\eqref{it:av}
  and~\ref{lem:bia}\eqref{it:ab}, we then have
  \begin{align*}
    \widehat T(\widehat \theta) - T(\widehat \theta) &=
    O(1)\tmax_{0\le j<J,k} \abs{\widehat \alpha_{0,0}(\widehat \theta) -
      \alpha_{0,0}(\widehat \theta)}, \abs{\widehat \beta_{j,k}(\widehat \theta) -
      \beta_{j,k}(\widehat \theta)}\\
    &= O_p(n^{-1/2}\log(n)^{1/2}).
  \end{align*}

  Finally, we note that the rates of convergence proved depend only
  upon the bounds assumed on the inputs. They therefore hold uniformly
  over models satisfying our assumptions.
\end{proof}

Next, we will prove our results on test coverage and detection rates.

\begin{proof}[Proof of \autoref{thm:ub}]
  We first note that part~\eqref{it:u0} is immediate from
  \autoref{thm:lim}\eqref{it:h0}. To prove part~\eqref{it:u1}, we
  consider separately the cases~\eqref{it:ua} and~\eqref{it:ub}.  In
  each case, we will prove that with probability tending to one, the
  event \(E_n\) implies
  \[T(\widehat \theta) \ge M_n'
  n^{-1/2}\log(n)^{1/2},\] for a fixed sequence \(M_n' \to \infty.\) The
  result will then follow from \autoref{thm:lim}\eqref{it:h1}.

  In case~\eqref{it:ua}, we note that arguing as in
  \autoref{thm:lim}, we may assume \(\Theta\) is bounded. Let
  \(\underline \Theta_n\) be an \(n^{-1/4}\)-net for \(\Theta,\) of size
  \(O(n^{p/4}),\) and \(\underline{\widehat \theta}\) be an element of
  \(\underline \Theta_n\) satisfying
  \[\underline{\widehat \theta} - \widehat \theta = O(n^{-1/4}).\]
  Using \autoref{lem:obs}\eqref{it:sl}, we have
  \[S_t(\widehat \theta) = S_t(\underline{\widehat \theta}) +
  O(n^{-1/4}),\]
  so on \(E_n,\)
  \[\norm{S(\underline{\widehat \theta})}_\infty \ge
  \norm{S(\widehat \theta)}_\infty - O(n^{-1/4}) \ge
  M_nn^{-1/4}\log(n)^{1/2}/2,\]
  for large \(n.\) We may thus assume further that \(\widehat \theta
  \in \underline \Theta_n.\)

  We then apply \autoref{lem:smt}, obtaining processes
  \(\widetilde S_t(\theta),\,\overline S_t(\theta),\) and times \(\tau_i.\) On the event
  \(E_n,\) for some point \(u \in [0,1],\) we have
  \[\abs{S_{u}(\widehat \theta)} \ge M_n n^{-1/4}\log(n)^{1/2}.\]
  We thus have \(u \in [\tau_i, \tau_{i+1})\) for some \(i < N-1,\) or
  \(u \in [\tau_{i}, \tau_{i+1}]\) for \(i = N-1.\) From
  \autoref{lem:smt}\eqref{it:lj}, with probability tending to one we
  also have \[\tau_{i+1}-\tau_i \ge 2^{1-J},\]
  and so there exists a point \(v \in [\tau_i+2^{-J},
  \tau_{i+1}-2^{-J}],\) \(\abs{u - v} \le 2^{-J}.\)

  We deduce that with probability tending to one,
\begin{align*}
  \mel \abs{\alpha_{0,0}(\widehat \theta)} + \tsum_{j=0}^{J-1}
  2^{j/2}\abs{\beta_{j,2^{-j}\lfloor 2^j v \rfloor}(\widehat
  \theta)}\\
&=\abs{\alpha_{0,0}(\widehat \theta)\varphi_{0,0}(v)}
  + \tsum_{0\le j<J,k} \abs{\beta_{j,k}(\widehat
  \theta)\psi_{j,k}(v)}\\
&\ge \abs{\alpha_{0,0}(\widehat \theta)\varphi_{0,0}(v)
  + \tsum_{0\le j<J,k} \beta_{j,k}(\widehat
  \theta)\psi_{j,k}(v)}\\
&=\abs{P_JS_{v}(\widehat \theta)},\\
  \intertext{writing the projection \(P_J\) in terms of the wavelet functions \(\psi_{j,k},\)}
&\ge \abs{S_{u}(\widehat
  \theta)} - \abs{S_{v}(\widehat \theta) - S_{u}(\widehat
  \theta)} - \abs{R_JS_{v}(\widehat \theta)}\\
&\ge \abs{S_{u}(\widehat \theta)} - \abs{\widetilde
  S_{v}(\widehat \theta) - \widetilde S_{u}(\widehat \theta)}
  - \abs{R_J\widetilde S_{v}(\widehat \theta)},\\
  \intertext{since \(\overline S_t(\widehat \theta)\) is constant within a distance \(2^{-J}\) of \(v,\)}
&\ge M_n n^{-1/4}\log(n)^{1/2}/2,
  \end{align*}
  using \hyperref[lem:smt]{Lemmas}~\ref{lem:smt}\eqref{it:hs}
  and~\eqref{it:sj}. We deduce that
  \begin{align*}
    T(\widehat \theta) &\ge \max(\abs{\alpha_{0,0}(\widehat \theta)}, \abs{\beta_{j,2^{-j}\lfloor 2^j v \rfloor}(\widehat \theta)} : j = 0, \dots, J-1),\\
                       &\ge 2^{-(J+3)/2} (\abs{\alpha_{0,0}(\widehat \theta)} + \tsum_{j=0}^{J-1}2^{j/2} \abs{\beta_{j,2^{-j}\lfloor 2^j v \rfloor}(\widehat \theta)})\\
                       &\ge M_n' n^{-1/2}\log(n)^{1/2},
  \end{align*}
  for a sequence \(M_n' \to \infty.\)

  In case~\eqref{it:ub}, on the event \(E_n,\) we likewise have
  \begin{align*}
    \mel \abs{\alpha_{0,0}(\widehat \theta)} + \tsum_{j=0}^{j_n-1}
    2^{j/2}\abs{\beta_{j,2^{-j}\lfloor 2^{j-j_n}k_n \rfloor}(\widehat \theta)}
    \\ &\ge\abs{P_{j_n}S_{2^{-j_n}k_n}(\widehat \theta)}\\
    &= 2^{j_n}\abs{\tint_{2^{-j_n}k_n}^{2^{-j_n}(k_n+1)} S_t(\widehat
      \theta)\,dt}\\
    &\ge M_n2^{j_n/2}n^{-1/2}\log(n)^{1/2},
  \end{align*}
  for some \(j_n = 0, \dots, J\) and \(k_n.\) The result then follows
  as in part~\eqref{it:u0}.
\end{proof}

Finally, we can prove our lower bound on detection rates.

\begin{proof}[Proof of \autoref{thm:lb}]
  In each case~\eqref{it:la} and~\eqref{it:lb}, we will reduce the
  statement to a known testing inequality. We will consider the model
  \[Y_i \coloneqq \delta_n^{1/2}n^{-1/2}2^{j_n}(B_{t_{i} \vee \tau} -
  B_\tau) + \varepsilon_i,\]
  where \(B_t\) is an adapted Brownian motion, the independent
  \(\mathcal F_{t_{i+1}}\)-measurable variables \(\varepsilon_i\) are
  standard Gaussian given \(\mathcal F_{t_i},\) \(\tau \in [0,1]\) is
  to be defined, and in case~\eqref{it:la} we set \(j_n \coloneqq J.\)
  It can be checked that this model satisfies our assumptions.

  Under \(H_0,\) we set \(\tau \coloneqq 1,\) so we have mean and variance
  functions \[\mu \coloneqq 0, \qquad \sigma^2 \coloneqq 1.\] Under \(H_1,\) we
  instead set \(\tau \coloneqq t_{m},\) where \(m \coloneqq \lfloor n(1 -
  2^{-j_n}) \rfloor.\) We then have
  \[S_t = \delta_n^{1/2}n^{-1/2}2^{j_n} (B_{t \vee
    \tau} - B_\tau),\] so in case~\eqref{it:la},
  \[\P[E_n] = \P[\norm{S}_\infty \ge \delta_n
  n^{-1/4}] \to 1.\]
  Similarly, in case~\eqref{it:lb},
  \[\P[E_n] \ge \P[2^{j_n/2}\abs{\tint_{1-2^{-j_n}}^1 S_t\,dt}
  \ge \delta_n n^{-1/2}] \to 1.\]

  It remains to show that no sequence of critical regions \(C_n\)
  can satisfy \(\lim \tsup_n \P[C_n] < 1\) under \(H_0,\) and
  \(\P[C_n] \to 1\) under \(H_1.\) We note that under
  \(H_0,\) we have \(Y \sim N(0, I),\) while under \(H_1,\) \(Y \sim
  N(0, I + \delta_n\Sigma),\) for a covariance matrix
  \[\Sigma_{k,l} = 0 \vee 2^{2j_n}(k \wedge l - m)/n^2.\]
  As \(\Sigma\) is non-negative definite, and has Frobenius norm
  \(O(1),\) the result follows from Lemma~2.1 of \citet{munk_lower_2010}.
\end{proof}


%% file: sdgft-tables.tex
\begin{table}
\centering
\begin{tabular}{lcccccc}
\toprule
\(n\) & \multicolumn{2}{c}{100} & \multicolumn{2}{c}{200} & \multicolumn{2}{c}{500} \\
\cmidrule(lr){2-3}\cmidrule(lr){4-5}\cmidrule(lr){6-7}
\(\alpha\) & 5\% & 10\% & 5\% & 10\% & 5\% & 10\% \\
\midrule
\addlinespace[1em]
\multicolumn{7}{l}{\em Constant volatility, null, \(\mu_t = 1\)}\\
\quad \(b_t = 0\) & 0.048 & 0.105 & 0.056 & 0.101 & 0.035 & 0.089 \\
\quad \(b_t = 2\) & 0.055 & 0.114 & 0.057 & 0.103 & 0.044 & 0.084 \\
\quad \(b_t = X_t\) & 0.056 & 0.101 & 0.041 & 0.093 & 0.037 & 0.092 \\
\quad \(b_t = 2 - X_t\) & 0.048 & 0.095 & 0.052 & 0.105 & 0.051 & 0.100 \\
\quad \(b_t = tX_t\) & 0.038 & 0.094 & 0.060 & 0.101 & 0.063 & 0.111 \\
\addlinespace[1em]
\multicolumn{7}{l}{\em Constant volatility, alternative, \(b_t = X_t\)}\\
\quad \(\sqrt{\mu_t} = 1 + X_t\) & 0.777 & 0.840 & 0.898 & 0.932 & 0.976 & 0.985 \\
\quad \(\sqrt{\mu_t} = 1 + \sin{5X_t}\) & 0.964 & 0.977 & 0.997 & 0.999 & 1.000 & 1.000 \\
\quad \(\sqrt{\mu_t} = 1 + X_t \exp{t}\) & 0.954 & 0.975 & 0.987 & 0.994 & 0.999 & 0.999 \\
\quad \(\sqrt{\mu_t} = 1 + X_t \sin{5t}\) & 0.851 & 0.908 & 0.970 & 0.982 & 0.994 & 0.995 \\
\quad \(\sqrt{\mu_t} = 1 + tX_t\) & 0.742 & 0.796 & 0.883 & 0.914 & 0.951 & 0.972 \\
\addlinespace[1em]
\multicolumn{7}{l}{\em Proportional volatility, null, \(\mu_t = X_t^2\)}\\
\quad \(b_t = 0\) & 0.062 & 0.119 & 0.044 & 0.090 & 0.043 & 0.087 \\
\quad \(b_t = 2\) & 0.073 & 0.120 & 0.056 & 0.106 & 0.043 & 0.081 \\
\quad \(b_t = X_t\) & 0.070 & 0.115 & 0.055 & 0.100 & 0.043 & 0.098 \\
\quad \(b_t = 2 - X_t\) & 0.053 & 0.085 & 0.055 & 0.100 & 0.034 & 0.081 \\
\quad \(b_t = tX_t\) & 0.070 & 0.106 & 0.062 & 0.123 & 0.045 & 0.106 \\
\addlinespace[1em]
\multicolumn{7}{l}{\em Proportional volatility, alternative, \(b_t = 2 - X_t\)}\\
\quad \(\mu_t = 1 + X_t^2\) & 0.602 & 0.673 & 0.700 & 0.766 & 0.844 & 0.884 \\
\quad \(\mu_t = 1\) & 0.832 & 0.871 & 0.927 & 0.951 & 0.979 & 0.991 \\
\quad \(\mu_t = 5\abs{X_t}^{3/2}\) & 0.580 & 0.669 & 0.672 & 0.760 & 0.854 & 0.902 \\
\quad \(\mu_t = 5\abs{X_t}\) & 0.896 & 0.932 & 0.963 & 0.974 & 0.995 & 0.998 \\
\quad \(\mu_t = (1+X_t)^2\) & 0.831 & 0.878 & 0.894 & 0.929 & 0.964 & 0.979 \\
\addlinespace[1em]
\bottomrule
\end{tabular}
\caption{Observed rejection probabilities for bootstrap test.}
\label{tab:results}
\end{table}

%% file: sdgft-proofs.tex
We now give proofs of our technical results, beginning with a
demonstration that our examples satisfy our assumptions.

\begin{lemma}
\label{lem:exa}
The examples in \autoref{sec:det} satisfy the conditions of
\autoref{ass:mod}.
\end{lemma}

\begin{proof}
  As in \autoref{sec:det}, we may assume our conditions on \(\mu,\)
  \(\sigma\) and \(\widehat \theta\) are satisfied. It thus remains to
  establish the conditions \eqref{eq:wsm}--\eqref{eq:ymb} for each of
  the examples.
\begin{description}
\item[Local volatility]

  By standard localisation arguments, we may assume condition
  \eqref{eq:wsm}, as well as that \(b_t,\,\mu_t^{-1} = O(1).\)
  Condition \eqref{eq:xha} is trivial. To establish condition \eqref{eq:ymb}, we make the
  decomposition
\[Y_i = (Z_{1,i} + Z_{2,i} + Z_{3,i})^2,\]
where
\begin{align*}¯
Z_{1,i} &\coloneqq \sqrt{n} \tint_{t_i}^{t_{i+1}} \sqrt{\mu_{t_i}} \, dB_t,\\
Z_{2,i} &\coloneqq \sqrt{n} \tint_{t_i}^{t_{i+1}} (\sqrt{\mu_t} - \sqrt{\mu_{t_i}}) \, dB_t,\\
Z_{3,i} &\coloneqq \sqrt{n} \tint_{t_i}^{t_{i+1}} b_t\,dt.
\end{align*}

We then have that, for \(p > 0,\)
\begin{align*}
Z_{1,i} \mid \F_{t_i} &\sim N(0, \mu_{t_i}), \\
\E[\abs{Z_{2,i}}^p \mid \F_{t_i}] &= O(n^{-p/2}),\\
\intertext{using Burkholder-Davis-Gundy, and}
\abs{Z_{3,i}} &= O(n^{-1/2}).
\end{align*}
The desired bounds follow using Cauchy-Schwarz.

\item[Jumps]

  By localisation, we may again assume \eqref{eq:wsm}, as well as the
  bounds
  \(b_t,\, \mu_t^{-1},\, \tint_\R 1 \wedge \abs{f_t(x)}^\beta \,dx =
  O(1).\) Condition \eqref{eq:xha} is again trivial. For
  \eqref{eq:ymb}, we write
\[Y_i = g_n(Z_{1,i} + Z_{2,i} + Z_{3,i} + Z_{4,i} + Z_{5,i}),\]
where
\begin{align*}
Z_{1,i} &\coloneqq \sqrt{n} \tint_{t_i}^{t_{i+1}} \sqrt{\mu_{t_i}}\, dB_t,\\
Z_{2,i} &\coloneqq \sqrt{n} \tint_{t_i}^{t_{i+1}} \tint_{A_t} f_t(x)\,\lambda(dx,dt),\\
Z_{3,i} &\coloneqq \sqrt{n} \tint_{t_i}^{t_{i+1}} \tint_{A_t^c} f_t(x)\,\lambda(dx,dt),\\
Z_{4,i} &\coloneqq \sqrt{n} \tint_{t_i}^{t_{i+1}} (\sqrt{\mu_t} - \sqrt{\mu_{t_i}})\,dB_t,\\
Z_{5,i} &\coloneqq \sqrt{n} \tint_{t_i}^{t_{i+1}} b_t\,dt,
\end{align*}
and \[A_t \coloneqq \{x \in \R : \abs{f_t(x)} \le n^{-1/2+\delta'}\},\] for
some sufficiently small \(\delta' > 0.\)

We then have that, for \(p = 2,4,16,\) and some \(\epsilon' > 0,\)
\begin{align*}
Z_{1,i} \mid \F_{t_i} &\sim N(0, \mu_{t_i}),\\
\E[\abs{Z_{2,i}}^p \mid \F_{t_i}] &= O(n^{-1/2 - \epsilon'}),\\
\intertext{by repeated application of Burkholder-Davis-Gundy,}
\E[Z_{1,i}Z_{2,i}\mid \F_{t_i}] &= 0,\\
\intertext{by It\=o's Lemma,}
\E[\abs{Z_{4,i}}^p \mid \F_{t_i}] &= O(n^{-p/2}),\\
\intertext{by Burkholder-Davis-Gundy, and}
\abs{Z_{5,i}} &= O(n^{-1/2}).
\end{align*}

Letting \(Z_i \coloneqq Z_{1,i} + Z_{2,i} + Z_{4,i} + Z_{5,i},\) and
\(Y_i' \coloneqq Z_i^2,\) we deduce that
\begin{align*}
  \E[Y_i' \mid \F_{t_i}] &= \mu_{t_i} + O(n^{-1/2}),\\
  \E[\abs{Y_i'}^2 \mid \F_{t_i}] &= \sigma_{t_i}^2 + \mu_{t_i}^2 + O(n^{-1/4}),\\
  \E[\abs{Y_i'}^8 \mid \F_{t_i}] &= O(1).
\end{align*}
It thus remains to control the difference between \(Y_i\) and
\(Y_i'.\)

We first note that for any \(p \ge 2,\) small enough \(q > 1,\) and
\(\alpha_n\) as in \eqref{eq:an},
\begin{align*}
  \E[\abs{Z_i}^p 1_{Z_i^2 > \alpha_n}]
  &= O(n^{-1/2}) + O(1)\E[\abs{Z_{1,i}}^p1_{Z_i^2 > \alpha_n}]\\
  &= O(n^{-1/2}) + O(1)\P[Z_i^2 > \alpha_n]^{1/q},\\
  \intertext{using H\"older's inequality,}
  &= O(n^{-1/2}),
\end{align*}
using a Gaussian tail bound, Markov's inequality, and that
\(\alpha_n\) grows super-logarithmically. We also have
\[\P[Z_{3,i} \ne 0\mid \F_{t_i}] = O(n^{-1/2-\delta'}),\]
using our assumptions on \(f_t(x).\) As \(\alpha_n\) grows
sub-polynomially, the required bounds follow also for \(Y_i.\)

\item[Microstructure noise]

  As before, by localisation we may assume condition \eqref{eq:wsm},
  as well as the boundedness of characteristic processes. To show \eqref{eq:xha},
  we first note that
\[\widehat X_{1,j} - X_{1,t_j} = W_{1,j} + W_{2,j},\]
where
\begin{align*}
W_{1,j} &\coloneqq n^{-1}\tsum_{i=0}^{n-1} \varepsilon_{i,nj+i},\\
W_{2,j} &\coloneqq n^{-1}\tsum_{i=0}^{n-1} (X_{1,t'_{nj+i}}-X_{1,t_j}).
\end{align*}

For \(p = 2, 4,\) we then have
\begin{align*}
\E[\abs{W_{1,j}}^p\mid\F_{t_j}] &= O(n^{-p/2}),\\
\intertext{using \autoref{lem:was}\eqref{it:em}, and}
\E[\abs{W_{2,j}}^p \mid \F_{t_j}] &= O(n^{-p/2}),
\end{align*}
using Burkholder-Davis-Gundy. We deduce that
\[\E[\abs{\widehat X_{1,j} - X_{1,t_j}}^p\mid\F_{t_j}] = O(n^{-p/2}).\]
By a similar method, decomposing \(\widehat X_{2,j}\) into a sum of
nuisance terms and martingale difference sequences, the same bound
holds also for \(\widehat X_{2,j}.\)

To show \eqref{eq:ymb}, we write
\[Y_j = 2(Z_{1,j} + Z_{2,j} + Z_{3,j})^2 - \pi^2 \widehat X_{2,j},\]
using integration by parts, where
\begin{align*}
Z_{1,j} &\coloneqq \bal \sqrt{n} \tint_{t_j}^{t_{j+1}} \sin(\pi n (t - t_j)) \sqrt{\mu_t} \, dB_t\\
&+ \pi n^{-1/2} \tsum_{i=0}^{n-1} \cos(\pi(i+\tfrac12)/n)\varepsilon_{nj+i},\eal\\
Z_{2,j} &\coloneqq \sqrt{n} \tint_{t_j}^{t_{j+1}} \sin(\pi n (t - t_j)) b_t\, dt,\\
Z_{3,j} &\coloneqq \pi (\bal n^{-1/2} \tsum_{i=0}^{n-1} \cos(\pi(i+\tfrac12)/n)X_{1,t'_{nj+i}}\\
&- n^{3/2} \tint_{t_j}^{t_{j+1}} \cos(\pi n(t-t_j)) X_{1,t}\,dt).\eal
\end{align*}

We then have
\begin{align*}
\E[2Z_{1,j}^2 \mid \F_{t_j}] &= \mu_{t_j} + \pi^2 X_{2,t_j} + O(n^{-1/2}),\\
\intertext{by direct computation,}
\E[4Z_{1,j}^4 \mid F_{t_j}] &= 3(\mu_{t_j} + \pi^2 X_{2,t_j})^2 + O(n^{-1/4}),\\
\intertext{using \autoref{lem:was}\eqref{it:gd},}
\E[\abs{Z_{1,j}}^\kappa \mid \F_{t_j}] &= O(1),\\
\intertext{using \autoref{lem:was}\eqref{it:em},}
\abs{Z_{2,j}} &= O(n^{-1/2}),\\
\intertext{and for \(p > 0,\)}
\E[\abs{Z_{3,j}}^p \mid \F_{t_j}] &= O(n^{-p}),
\end{align*}
using \autoref{lem:was}\eqref{it:em}. The desired results follow.

\item[Stochastic volatility]

  Using integration by parts, we can make a decomposition
\[Y_j = 2(Z_{1,j} + Z_{2,j} + Z_{3,j} + Z_{4,j})^2 - 2\pi^2 \widehat X_{2,j}^2,\]
where
\begin{align*}
  Z_{1,j} &\coloneqq \bal \sqrt{n}\tint_{t_j}^{t_{j+1}} \sin (\pi n(t-t_j)) \sqrt{\mu_t}\,dB_t'\\
          &+ \bal \pi n^{-1/2} \tsum_{i=0}^{n-1} \cos(\pi(i+\tfrac12)/n) \times\\
            &(\widetilde X_{2,nj+i} - \E[\widetilde X_{2,nj+i} \mid \F_{t'_{nj+i}}]),\eal\eal\\
  Z_{2,j} &\coloneqq \sqrt{n}\tint_{t_j}^{t_{j+1}} \sin (\pi n(t-t_j)) b'_t\,dt\\
  Z_{3,j} &\coloneqq \pi (\bal n^{-1/2} \tsum_{i=0}^{n-1} \cos(\pi(i + \tfrac12)/n) X_{2,t'_{nj+i}}\\
&- n^{3/2} \tint_{t_j}^{t_{j+1}} \cos(\pi n(t - t_j)) X_{2,t}\,dt),\eal\\
Z_{4,j} &\coloneqq \bal\pi n^{-1/2} \tsum_{i=0}^{n-1} \cos(\pi(i+\tfrac12)/n) \times \\
&(\E[\widetilde X_{2,nj+i} \mid \F_{t'_{nj+i}}] - X_{2,t'_{nj+i}}).\eal
\end{align*}
The desired results then follow similarly to the microstructure noise
example.\qedhere
\end{description}
\end{proof}

Next, we give some standard exponential moment bounds on stochastic
integrals.

\begin{lemma}
  \label{lem:exp}
  Let \((\Omega, \mathcal F, (\mathcal F_t)_{t \in [0, \infty]}, \P)\)
  be a filtered probability space, with adapted Brownian motion \(B_t
  \in \R,\) and Poisson random measure \(\lambda(dx,dt)\) having
  compensator \(dx\,dt.\)
  \begin{enumerate}
  \item \label{it:ce} Let \(c_t \in \R\) be a predictable process, and
    define
    \[\rho^2_t \coloneqq \tint_0^t c_s^2\,ds.\]
    If \(\rho^2_\infty < \infty\) almost surely, then for \(u \in
    \R,\) the stochastic integral
    \[W_t \coloneqq \tint_0^t c_s\,dB_s\] satisfies
    \[\E[\exp(uW_\infty - \tfrac12 u^2 \rho_\infty^2) \mid \mathcal F_0] \le
    1.\]
  \item \label{it:se} In the setting of part~\eqref{it:ce}, if
    \(\rho^2_\infty \le R\) almost surely, we further have
    \[\E[\exp(u\tsup_{t \ge 0} \abs{W_t} - u^2R)\mid \mathcal F_0] = O(1).\]
  \item \label{it:de}
    Let \(t \ge 0,\) and \(f_s(x) \in \R\) be a predictable function,
    with \(\abs{f_s(x)} \le 1,\)
    \[\tint_0^t \tint_\R \abs{f_s(x)}\,dx\,ds = O(R),\]
    for some \(R \in (0, 1).\) Then for \(u \le
    -\log(R),\) the stochastic integral
    \[W_t \coloneqq \tint_0^t \tint_\R f_s(x) \,(\lambda(dx,ds) -
    dx\,ds)\]
    satisfies
    \[\E[\exp(u\tsup_{s \in [0,t]}\abs{W_s}) \mid \mathcal F_0] = O(1).\]
  \end{enumerate}
\end{lemma}

\begin{proof}
  We begin with part~\eqref{it:ce}.  By localisation and martingale
  convergence, we have
  \[W_t \overset{\mathrm{a.s.}}\to W_\infty.\] The Dol\'eans-Dade
  exponential
  \[\mathcal E(uW)_t \coloneqq \exp(uW_t - \tfrac12 u^2 \rho_t^2)\]
  is a non-negative local martingale, so a supermartingale. Hence by
  Fatou's lemma,
  \[\E[\mathcal E(uW)_\infty\mid \mathcal F_0] \le \textstyle \lim\inf_{t\to\infty}
  \E[\mathcal E(uW)_t \mid \mathcal F_0] \le \E[\mathcal E(uW)_0 \mid
  \mathcal F_0] = 1.\]

  To show part~\eqref{it:se}, we have
  \[\E[\mathcal E(uW)_\infty^2 \mid \mathcal F_0] \le \exp(u^2R)\E[\mathcal
  E(2uW)_\infty \mid \mathcal F_0] \le \exp(u^2R),\]
  so \(\mathcal E(uW)_t\) is a true martingale. We deduce that
  \begin{align*}
    \mel\E[\exp(u \tsup_{t \ge 0} \abs{W_t} - u^2R) \mid \mathcal F_0]\\
    &\le \exp(-\tfrac12u^2R)\E[\tsup_{t \ge 0} \mathcal E(uW)_t + \tsup_{t \ge 0} \mathcal E(-uW)_t \mid \mathcal F_0]\\
    &= O(1)\exp(-\tfrac12u^2R)\E[\mathcal E(uW)_\infty^2 + \mathcal E(-uW)_\infty^2 \mid \mathcal F_0]^{1/2},\\
    \intertext{using Doob's martingale inequality,}
    &= O(1).
  \end{align*}

  We now prove part~\eqref{it:de}, noting we may assume \(u \ge 0.\)
  Defining the variation martingale
  \[M_t \coloneqq \tint_0^t \tint_\R \abs{f_s(x)}\,(\lambda(dx, ds) - dx\,ds),\]
  we then have
  \[\tsup_{s \in [0,t]} \abs{W_s} \le M_t + O(R),\]
  so it suffices to bound \(M_t.\) Let \(M_t^c\) denote the continuous
  part of \(M_t,\) and \(\Delta M_t\) its jump at time \(t.\) Then for
  \(v \ge 0,\) the Dol\'eans-Dade exponential
  \[\mathcal E(vM)_t \coloneqq \exp(vM_t^c)\tprod_{s\le t}(1 + v\Delta
  M_s)\] is a non-negative local martingale, so a supermartingale.

  Furthermore, since \(u \Delta M_s \in [0, -\log(R)],\) we have
  \[\exp(u\Delta M_s) \le 1 + cu\Delta M_s,\]
  where the constant \(c \coloneqq (1-R^{-1})/\log(R).\) We deduce
  that
  \begin{align*}
    \E[\exp(uM_t) \mid \mathcal F_0] &= O(1)\E[\exp((c-1)uM_t^c
    +
    uM_t) \mid \mathcal F_0]\\
    &= O(1)\E[\exp(cuM_t^c + u\tsum_{s\le t} \Delta M_s) \mid
    \mathcal F_0]\\
    &= O(1)\E[\mathcal E(cuM)_t \mid \mathcal F_0]\\
    &= O(1)\mathcal E(cuM)_0\\
    &= O(1).\qedhere
  \end{align*}
\end{proof}

We may now prove our central limit theorem for martingale
differences. Our argument uses a Skorokhod embedding, as in
\citet{mykland_embedding_1995} or \citet{obloj_skorokhod_2004}, for
example.

\begin{proof}[Proof of \autoref{lem:was}]
  We begin with a Skorokhod embedding, allowing us to consider the
  variables \(X_i\) as stopped Brownian motions on an extended
  probability space. Our argument proceeds by induction on a variable
  \(k = 0, \dots, n.\)
  
  We claim that for \(i = 0, \dots, k-1,\) on an extended probability
  space, we can construct processes \((B_{i,t})_{t \in [0,\infty)},\)
  which are Brownian motions given the \(\sigma\)-algebra \(\widetilde
  {\mathcal F}_{i}\) generated by \(\mathcal F_{i}\) and
  \(B_0,\dots,B_{i-1},\) and are independent of \(\mathcal F\) given
  \(\mathcal F_{i+1}.\) We further claim we can construct variables
  \(\tau_i \in [0, \infty),\) which are stopping times in the natural
  filtrations \(\mathcal G_{i,t}\) of the \(B_{i,t},\) so that \(X_i =
  B_{i,\tau_i}.\)

  For \(k = 0,\) the claim is trivial; we will show that if the claim
  holds for \(k,\) it holds also for \(k+1.\) By Skorokhod embedding,
  on a further-extended probability space, we can construct a process
  \(\widetilde B_{k,t}\) which is a Brownian motion given \(\widetilde
  {\mathcal F}_{k},\) and a variable \(\widetilde \tau_k\) which is
  a stopping time in the natural filtration of \(\widetilde B_{k,t},\)
  such that the variable \(\widetilde X_k \coloneqq \widetilde
  B_{k,\widetilde \tau_k}\) is distributed as \(X_k\) given
  \(\widetilde{\mathcal F}_{k}.\)

  Since the stopped process \((\widetilde B_{k,t\wedge \widetilde
    \tau_k})_{t\in[0,\infty)}\) is continuous and eventually constant,
  the pair \(((\widetilde B_{k,t\wedge \widetilde
    \tau_k})_{t\in[0,\infty)},\widetilde \tau_k)\) takes values in a
  Polish space. We can thus define the regular conditional
  distribution \(\mathbb Q_k(x)\) of \(((\widetilde
  B_{k,t\wedge\widetilde \tau_k})_{t\in[0,\infty)},\widetilde
  \tau_k)\) given \(\widetilde X_k = x\) and \(\widetilde{\mathcal
    F}_{k}.\) On a further-extended probability space, we can then
  generate a pair \(((B_{k,t\wedge \tau_k})_{t\in[0,\infty)},\tau_k)\)
  with distribution \(\mathbb Q_k(X_k)\) given \(\widetilde{\mathcal
    F}_{k}\) and \(X_k,\) independent of \(\mathcal F\) given
  \(\mathcal F_{k+1}.\)

  We deduce that the triplet \(((B_{k,t\wedge
    \tau_k})_{t\in[0,\infty)},\tau_k,X_k)\) is distributed as the
  triplet \(((\widetilde B_{k,t\wedge \widetilde
    \tau_k})_{t\in[0,\infty)},\widetilde \tau_k,\widetilde X_k)\)
  given \(\widetilde{\mathcal F}_{k},\) and hence \(B_{k,t \wedge
    \tau_k}\) and \(\tau_k\) satisfy the conditions of our claim. It
  remains to define \(B_{k,t}\) for \(t > \tau_k;\) we set
  \[B_{k,t+\tau_k} \coloneqq B_{k,\tau_k} + B_{k,t}', \qquad t \ge
  0,\]
  for an independent Brownian motion \(B_{k,t}'.\) We then conclude
  that \(B_{k,t}\) and \(\tau_k\) satisfy the conditions of our claim;
  by induction, the claim thus holds for \(k = n.\)

  Next, we will show we can realise the sums \(\tsum_{i=0}^{n-1}
  c_iX_i\) as integrals against a common Brownian motion. Define a
  process
  \[B_t \coloneqq \tsum_{j=0}^{n-2} B_{j,T(j,t) \wedge \tau_j} + B_{n-1,T(n-1,t)},\]
  where the variables
  \[T(j,t) \coloneqq 0 \vee (t - \tsum_{i=0}^{j-1} \tau_i).\]
  We will show that \(B_t\) is a Brownian motion with respect to a
  suitable filtration \(\mathcal G_t,\) and that the sums
  \(\tsum_{i=0}^{n-1} c_iX_i\) can be written as stochastic integrals
  against \(B_t.\)

  For fixed \(j = 0, \dots, n-1,\) the \(\sigma\)-algebras
  \[\widetilde{\mathcal G}_{j,t} \coloneqq \sigma(\widetilde{\mathcal
    F}_{j}, \mathcal G_{j,t})\]
  form a filtration in \(t \ge 0,\) and the variables \(T(j,t)\) are
  \(\widetilde{\mathcal G}_{j,t}\)-stopping times.  For fixed \(t \ge
  0,\) we can thus define the \(\sigma\)-algebras
  \(\widetilde{\mathcal G}_{j,T(j,t)},\) which form a filtration in
  \(j = 0, \dots, n-1,\) and the variables
  \[j(t) \coloneqq \tmax\{j=0, \dots, n-1:\tsum_{i=0}^{j-1} \tau_i \le
  t\},\]
  which are \(\widetilde{\mathcal G}_{j,T(j,t)}\)-stopping times.

  We can then define the \(\sigma\)-algebras
  \[\mathcal G_t \coloneqq \widetilde{\mathcal G}_{j(t),T(j(t),t)},\]
  which form a filtration in \(t,\) and check that the process \(B_t\)
  is a \(\mathcal G_t\)-Brownian motion. We conclude that given
  \(\mathcal F_{i}\)-measurable variables \(c_i,\) the sums
  \begin{equation}
    \label{eq:cx}
    \tsum_{i=0}^{n-1} c_iX_i = \tint_0^\infty f_c(t)\,dB_t,
  \end{equation}
  where the \(\mathcal G_t\)-predictable integrands
  \[f_c(t) \coloneqq \tsum_{j=0}^{n-1} c_j 1_{(0, \tau_j]}(T(j,t)).\]

  In part~\eqref{it:gd}, we consider the case \(c_i = 1,\) and obtain
  \[\tsum_{i=0}^{n-1}X_i = B_\nu, \qquad \nu \coloneqq \tsum_{i=0}^{n-1} \tau_i.\]
  Defining the
  random variables
  \[\xi \coloneqq B_1, \qquad \eta \coloneqq B_\nu - B_1,\]
  we then have \(\tsum_{i=0}^{n-1}X_i = \xi + \eta,\) and \(\xi \sim
  N(0,1)\) given \(\mathcal F_0.\)

  Furthermore, using Burkholder-Davis-Gundy, we have
  \[\E[\abs{\eta}^{4\kappa} \mid \F_0] = O(1) \E[\abs{\nu - 1}^{2 \kappa} \mid \F_0],\]
  while using \autoref{lem:exp}\eqref{it:ce},
  \[\E[\exp(u\eta - \tfrac12 u^2 \abs{\nu - 1}) \mid
  \mathcal F_0] \le 1.\]
  It thus remains to bound the distance of \(\nu\) from 1.

  For \(j = 0, \dots, n,\) define the \(\widetilde{\mathcal
    F}_j\)-martingale
  \[V_{j} \coloneqq \tsum_{i=0}^{j-1} (\tau_i - \E[\tau_i \mid
  \widetilde{\mathcal F}_{i}]),\]
  and the total mean
  \[\overline \nu \coloneqq
  \tsum_{i=0}^{n-1} \E[\tau_i \mid \widetilde{\mathcal F}_{i}].\]
  We then have
  \[\abs{\nu - 1} \le \abs{V_n} + \abs{\overline \nu - 1};\]
  we will show that both terms on the right-hand side are small.

  We first obtain
  \begin{align}
    \notag \E[\abs{V_n}^{2\kappa}\mid\mathcal F_0]
    &=
      O(1)\E[(\tsum_{i=0}^{n-1} \abs{\tau_i - \E[\tau_i\mid
      \widetilde{\mathcal F}_{i}]}^2)^{\kappa}\mid\mathcal F_0],\\
    \intertext{by Burkholder-Davis-Gundy,}
    \notag &=
      O(n^{\kappa-1})\tsum_{i=0}^{n-1}\E[\abs{\tau_i -
      \E[\tau_i\mid\widetilde{\mathcal
      F}_{i}]}^{2\kappa}\mid\mathcal F_0],\\
    \intertext{by Jensen's inequality,}
    \notag &=
      O(n^{\kappa-1})\tsum_{i=0}^{n-1}\E[\abs{\tau_i}^{2\kappa}
      \mid\mathcal F_0]\\
    \notag &=
      O(n^{\kappa-1})\tsum_{i=0}^{n-1}\E[\abs{X_i}^{4\kappa}\mid\mathcal F_0],
      \intertext{by Burkholder-Davis-Gundy and Doob's martingale inequality,}
    \label{eq:nm} &= O(n^{-\kappa}).
  \end{align}

  We also have that
  \begin{align*}
    \overline \nu
    &= \tsum_{i=0}^{n-1}\E[X_{i}^2 \mid \widetilde{\mathcal F}_{i}],\\
    \intertext{by Ito's isometry,}
    &= \tsum_{i=0}^{n-1}\E[X_i^2 \mid \mathcal F_{i}],
  \end{align*}
  as the \(B_{j,t}\) are independent of \(\mathcal F\) given
    \(\mathcal F_{j+1}.\) We deduce that
  \[\E[\abs{\nu - 1}^{2\kappa} \mid \mathcal F_0] = O(n^{-\kappa}),\]
  as required.

  In part~\eqref{it:em}, we again apply Burkholder-Davis-Gundy and
  \autoref{lem:exp}\eqref{it:ce} to the sums~\eqref{eq:cx}.  We claim
  that
  \[\tsup_c \tint_0^\infty f_c^2(t)\,dt \le A + M,\]
  for terms \(A\) and \(M\) as in the statement of the
  \hyperref[lem:was]{Lemma}, so
  \[\tsup_c \E[\abs{\upsilon_c}^{4\kappa}\mid\F_0] = O(1),\] and
  \[\tsup_c \E[\exp(u \upsilon_c - \tfrac12 u^2 (A + M))
  \mid \mathcal F_0] \le 1.\]
  It thus remains to prove the claim.

  As before, we have
  \[\abs{\nu} \le \abs{V_n} + \abs{\overline \nu},\]
  and
  \begin{align*}
    \overline \nu &= \tsum_{i=0}^{n-1}\E[X_i^2 \mid \mathcal F_{i}]\\
    &= O(n^{1-1/2\kappa})(\tsum_{i=0}^{n-1} \E[\abs{X_i}^{4\kappa} \mid \mathcal
      F_i])^{1/2\kappa},\\
    \intertext{by Jensen's inequality,}
    &=O(1).
  \end{align*}
  For any random variables \(c_i
  = O(1),\) we deduce that
  \[\tint_0^\infty f_c^2(t)\,dt = \tsum_{i=0}^{n-1} c_i^2 \tau_i =
  O(\nu) = O(1 + \abs{V_n}).\]
  The claim thus holds for terms \(A = O(1),\) \(M = O(\abs{V_n}),\)
  and we further have that \(M\) satisfies~\eqref{eq:ms},
  using~\eqref{eq:nm}.
\end{proof}

We next prove our result on combining exponential moment bounds.

\begin{proof}[Proof of \autoref{lem:tec}]
  We first note that by rescaling the \(X_i,\) we may assume \(r_n =
  1.\) Then on an extended probability space, let \(\xi\) be standard
  Gaussian, independent of \(\mathcal F.\) For any \(R > 0,\) we have
  \begin{align*}
    \mel \E[\exp(X_i^2 / 4R)1_{M \le R} ]\\
    &= \E[\E[\exp((2R)^{-1/2} X_i \xi)1_{M \le R}\mid\mathcal F]]\\
    &\le \E[\exp(\xi^2 / 4) \E[\exp((2R)^{-1/2} \xi X_i - \tfrac12
    (2R)^{-1} \xi^2 M)1_{M \le R} \mid \xi]]\\
    &= O(1) \E[\exp(\xi^2 / 4) ]\\
    &= O(1).
  \end{align*}
  We deduce that, for any \(R > 0,\)
  \[\E[\tmax_i\exp(X_i^2 / 4R)1_{M \le R}] \le
  \tsum_i \E[\exp(X_i^2 / 4R)1_{M \le R}] = O(n),\]
  so \[\tmax_i \,\abs{X_i}1_{M \le R} = O_p(R^{1/2} \log(n)^{1/2}).\]
  Since \(M =
  O_p(1),\) we conclude that \(\tmax_i \,\abs{X_i} =
  O_p(\log(n)^{1/2}).\)
\end{proof}

We continue with a proof of our moment bounds on the \(Y_i\) and
\(Z_i(\theta).\)

\begin{proof}[Proof of \autoref{lem:obs}]
  To show part~\eqref{it:yl}, we note that the functions \(\mu\) and
  \(\sigma^2\) are locally Lipschitz, \(\sigma^2\) is positive, and
  \(\theta,\) \(t_{i}\) and \(\widehat X_i\) are bounded. We may
  thus restrict the functions \(\mu\) and \(\sigma^2\) to a compact
  set, on which \(\mu\) and \(\sigma^2\) are \(C^1,\) and
  \(1/\sigma^2\) is bounded. We deduce that part~\eqref{it:yl} holds;
  by a similar argument, part~\eqref{it:sl} holds also.

  To show part~\eqref{it:zm}, we then have
  \[\E[\abs{Z_i(\theta)}^{4 + \varepsilon} \mid \mathcal
  F_{t_{i}}] = O(1)\E[1 + \abs{Y_i}^{4 + \varepsilon} \mid \mathcal
  F_{t_{i}}] = O(1),\] and
  \[Z_i(\theta) = (Y_i - \mu(\theta,t_{i},
  X_{t_{i}}))/\sigma(\theta, t_{i},X_{t_{i}}) + \gamma_i,\] for a term
  \[\gamma_i = O(1)(1 + \abs{Y_i})\norm{\widehat X_i - X_{t_{i}}}.\]

  Using Cauchy-Schwarz, we obtain
  \begin{align*}
    \E[\gamma_i \mid \mathcal F_{t_{i}}] &= O(n^{-1/2}),\\
    \E[\gamma_i^2 \mid \mathcal F_{t_{i}}] &= O(n^{-1/2}).
  \end{align*}
  We conclude that
  \begin{align*}
    \E[Z_i(\theta) \mid \mathcal F_{t_{i}}] &=
    S_{t_{i}}(\theta) + \E[\gamma_i \mid \F_{t_i}],\\
    &= S_{t_{i}}(\theta) + O(n^{-1/2}),
  \end{align*}
  and under \(H_0,\) using Cauchy-Schwarz, also
  \begin{align*}
    \E[Z_i(\theta_0)^2 \mid \mathcal F_{t_{i}}] 
    &= 1 + O(1)(\E[\gamma_i^2 \mid \F_{t_i}]^{1/2} + \E[\gamma_i^2 \mid \F_{t_i}]),\\
    &= 1 + O(n^{-1/4}).
  \end{align*}
  
  To show part~\eqref{it:ys}, we define the random variables
  \[R_k \coloneqq m^{-1} \tsum_{i=ns_k}^{ns_{k+1}-1} ((Y_i-\E[Y_i \mid
  \mathcal F_{t_i}])^2 - \Var[Y_i \mid \mathcal F_{t_i}]),\]
  where \(m \coloneqq n(s_{k+1}-s_k).\) \(R_k\) is then an average of
  \(m\) terms of a martingale difference sequence, whose conditional
  variances are bounded. We deduce that
  \[\E[R_k^2] = O(m^{-1}) = O(n^{-1/2}),\]
  and so
  \begin{align*}
    \E[(\tmax_k n^{-1/2} \tsum_{i=ns_k}^{ns_{k+1}-1} Y_i^2)^2]
    &= O(1)\E[1 + \tmax_k R_k^2]\\
    &= O(1)(1 + \tsum_k \E[R_k^2])\\
    &= O(1).
  \end{align*}
  The desired result follows.
\end{proof}

Finally, we prove our result on the behaviour of the processes
\(S_t(\theta)\) under \autoref{ass:ito}.

\begin{proof}[Proof of \autoref{lem:smt}]
  We begin by defining the processes \(\widetilde S_t(\theta),\)
  \(\overline S_t(\theta),\) and times \(\tau_i.\)
  We can split the process \(\mu_t\) into parts
  \[\mu_t = \widetilde \mu_t + \overline \mu_t,\]
  where \(\widetilde \mu_t\) is a process with jumps of size at most
  \(n^{-1/4}\log(n)^{1/2},\) and \(\overline \mu_t\) is an orthogonal
  pure-jump process with jumps of size at least
  \(n^{-1/4}\log(n)^{1/2}.\) We can similarly define terms
  \(\widetilde X_t,\) \(\overline X_t.\)

  Let \(\tau_1 < \dots < \tau_{N-1}\) denote the times at which
  \(\overline \mu_t\) or \(\overline X_t\) jump, and set \(\tau_0 \coloneqq 0,\)
  \(\tau_N \coloneqq 1.\) We can then decompose the processes
  \begin{equation}
    \label{eq:sd}
    S_t(\theta) = \widetilde S_t(\theta) + \overline S_t(\theta),
  \end{equation}
  where
  \[\overline S_t(\theta) \coloneqq \tsum_{\tau_i \le t} \Delta
  S_{\tau_i}(\theta),\]
  letting \(\Delta S_{t}(\theta)\) denote the jump in \(S_t(\theta)\)
  at time \(t,\) and \(\widetilde S_t(\theta)\) is then defined by~\eqref{eq:sd}.

  To prove part~\eqref{it:hs}, we first note that the model functions
  \(\mu\) and \(\sigma^2\) are continuously differentiable in \(t,\)
  twice continuously differentiable in \(X,\) and \(\sigma^2\) is
  positive. By It\=o's lemma, we can thus write
  \[d\widetilde S_t(\theta) = \widetilde b_t(\theta)\,dt + \widetilde
  c_t(\theta)^T\,dB_t + \tint_\R \widetilde
  f_t(x,\theta)\,(\lambda(dx, dt) - dx\,dt),\]
  for integrators \(B_t\) and \(\lambda(dx,dt)\) given by
  \autoref{ass:ito}, predictable processes \(\widetilde b_t(\theta),\,
  \widetilde c_t(\theta),\) and predictable functions \(\widetilde
  f_t(x,\theta).\) Since \(\theta,\) \(t,\) \(\mu_t\) and \(X_t\) are
  bounded, we also have \(\widetilde b_t(\theta),\,\widetilde
  c_t(\theta) = O(1),\) \(\widetilde f_t(x,\theta) = O(
  n^{-1/4}\log(n)^{1/2}),\) and \(\tint_\R \abs{\widetilde
    f_t(x,\theta)}\,dx = O(1).\)

  To bound the size of changes in \(\widetilde S_t(\theta),\) we will consider
  the variables
  \[M_k(\theta) \coloneqq \tsup_{t \in I_k} \abs{\widetilde S_t(\theta) -
    \widetilde S_{2^{-J}k}(\theta)},\]
  where the intervals \(I_k \coloneqq 2^{-J}[k, k+1].\) We have
  \[M_k(\theta) \le \tsum_{i=0}^{q+2} M_{k,i}(\theta),\]
  for terms
  \[
  M_{k,i}(\theta) \coloneqq
  \begin{cases}
    \tsup_{t \in I_k}
    \abs{\tint_{2^{-J}k}^t \widetilde b_t(\theta)\,dt},& i = 0,\\
    \tsup_{t \in I_k} \abs{\tint_{2^{-J}k}^t \widetilde
      c_{i,t}(\theta)\,dB_{i,t}},&
    i = 1, \dots, q+1,\\
    \tsup_{t \in I_k} \abs{\tint_{2^{-J}k}^t \tint_\R \widetilde f_t(x,
      \theta)\, (\lambda(dx,dt) - dx\,dt)},& i = q+2.
  \end{cases}\]
  
  In each case \(i = 0, \dots, q+2,\) we will bound the maximum
  \[\widetilde M_{i} \coloneqq \tmax_{k,\theta \in \underline \Theta_n} M_{k,i}(\theta).\]
  From the definitions, we have
  \[\widetilde M_0 = O(n^{-1/2}).\]
  For \(i = 1, \dots, q+1,\) we use \autoref{lem:exp}\eqref{it:se},
  obtaining that 
  \[\E[\exp(un^{1/4}M_{k,i}(\theta) - u^2R)] = O(1),\]
  for all \(u \in \R,\) and some fixed \(R > 0.\) Using
  \autoref{lem:tec}, we deduce that
  \[\widetilde M_i = O_p(n^{-1/4}\log(n)^{1/2}).\]

  Finally, using \autoref{lem:exp}\eqref{it:de}, for small enough
  \(\varepsilon' > 0\) we have
  \[\E[\exp(\varepsilon'n^{1/4}\log(n)^{1/2}M_{k,q+2}(\theta))] = O(1).\]
  We deduce that
  \begin{align*}
    \E[\exp(\varepsilon' n^{1/4}\log(n)^{1/2}\widetilde M_{q+2})]
    &\le \tsum_{k,\theta \in
    \underline
    \Theta_n}\E[\exp(\varepsilon' n^{1/4}\log(n)^{1/2}M_{k,q+2}(\theta))]\\
    &=
  O(n^{\kappa+1/2}),
  \end{align*}
  and so \[\widetilde M_{q+2} = 
  O_p(n^{-1/4}\log(n)^{1/2}).\]
  We conclude that the random variable
  \[
  M \coloneqq \tmax_{k,\theta \in \underline \Theta_n} M_k(\theta) \le
  \tsum_{i=0}^{q+2} \widetilde M_i = O_p(n^{-1/4}\log(n)^{1/2}).\]
  Part~\eqref{it:hs} then follows trivially.

  To show part~\eqref{it:sj}, for \(s, t \in [0,1],\) define the
  translated processes
  \[\widetilde S_s^{(t)}(\theta) \coloneqq \widetilde S_s(\theta) -
  \widetilde S_t(\theta).\] We then have
  \begin{align*}
    R_J\widetilde S_t(\theta)
    &= R_J\widetilde S_t^{(t)}(\theta),\\
    \intertext{since wavelets are orthogonal to constant functions,}
    &= -P_J\widetilde S_t^{(t)}(\theta),\\
    \intertext{since \(\widetilde S_t^{(t)}(\theta) = 0,\)}
    &= -\tsum_k \varphi_{J,k}(t) \tint_0^1 \varphi_{J,k}(s) (\widetilde S_s(\theta) -
      \widetilde S_t(\theta))\,ds\\
    &= O(M),
  \end{align*}
  using the compact support of \(\varphi.\) The desired result
  follows.

  To show part~\eqref{it:lj}, we first note that the processes
  \(\overline S_t(\theta)\) are constant on the intervals \([\tau_i,
  \tau_{i+1})\) and \([\tau_{N-1},\tau_N].\) Setting \(\tau_i = 1\)
  for \(i > N,\) we then have
  \begin{align*}
    \mel\P[\exists\,i \le n^{1/4} :\tau_i \le 1 - \delta_n,\,\tau_{i+1}<\tau_i +
    \delta_n]\\
    &\le \tsum_{i=0}^{\lfloor n^{1/4} \rfloor} \P[\tau_i \le 1 - \delta_n]\P[\tau_{i+1} <
      \tau_i+\delta_n \mid \tau_i \le 1 - \delta_n]\\
    &=o(n^{-1/4})\tsum_{i=0}^{\lfloor n^{1/4} \rfloor} \P[\tau_i \le 1 - \delta_n],\\
    \intertext{using \autoref{ass:ito}, and that \(\tau_i\) is a stopping time,}
    &=o(1).
  \end{align*}
  Similarly, we have 
  \[\P[\exists\,i: \tau_i \in (1-\delta_n,1)],\, \P[N > n^{1/4}] = o(1).\]
  The desired result follows.
\end{proof}


%% file: sdgft.bbl
\begin{thebibliography}{35}
\providecommand{\natexlab}[1]{#1}
\providecommand{\url}[1]{\texttt{#1}}
\expandafter\ifx\csname urlstyle\endcsname\relax
  \providecommand{\doi}[1]{doi: #1}\else
  \providecommand{\doi}{doi: \begingroup \urlstyle{rm}\Url}\fi

\bibitem[{A\"it-Sahalia}(1996)]{ait-sahalia_testing_1996}
{A\"it-Sahalia}  Y.
\newblock Testing continuous-time models of the spot interest rate.
\newblock \emph{Review of Financial Studies}, 9\penalty0 (2):\penalty0
  385--426, 1996.

\bibitem[{A\"it-Sahalia} and Park(2012)]{ait-sahalia_stationarity-based_2012}
{A\"it-Sahalia}  Y and Park  J Y.
\newblock Stationarity-based specification tests for diffusions when the
  process is nonstationary.
\newblock \emph{J. Econometrics}, 169\penalty0 (2):\penalty0 279--292, 2012.

\bibitem[Andersen et~al.(2001)Andersen, Bollerslev, Diebold, and
  Ebens]{andersen_distribution_2001}
Andersen  T G, Bollerslev  T, Diebold  F X, and Ebens  H.
\newblock The distribution of realized stock return volatility.
\newblock \emph{Journal of Financial Economics}, 61\penalty0 (1):\penalty0
  43--76, 2001.

\bibitem[{Barndorff-Nielsen} and
  Shephard(2002)]{barndorff-nielsen_econometric_2002}
{Barndorff-Nielsen}  O E and Shephard  N.
\newblock Econometric analysis of realized volatility and its use in estimating
  stochastic volatility models.
\newblock \emph{J. R. Stat. Soc. Ser. B Stat. Methodol.}, 64\penalty0
  (2):\penalty0 253--280, 2002.

\bibitem[{Barndorff-Nielsen} and
  Veraart(2009)]{barndorff-nielsen_stochastic_2009}
{Barndorff-Nielsen}  O E and Veraart  A.
\newblock Stochastic volatility of volatility in continuous time.
\newblock \emph{CREATES research paper}, 25, 2009.

\bibitem[Bibinger et~al.(2015)Bibinger, Jirak, and
  Vetter]{bibinger_nonparametric_2015}
Bibinger  M, Jirak  M, and Vetter  M.
\newblock Nonparametric change-point analysis of volatility.
\newblock \emph{arXiv preprint arXiv:1502.00043}, 2015.

\bibitem[Bull(2016)]{bull_software_2016}
Bull  A D.
\newblock Software for ``{S}emimartingale detection and goodness-of-fit
  tests''.
\newblock \url{http://dx.doi.org/10.17863/CAM.125}, 2016.

\bibitem[Chen et~al.(2015)Chen, Zheng, and Pan]{chen_asymptotically_2015}
Chen  Q, Zheng  X, and Pan  Z.
\newblock Asymptotically distribution-free tests for the volatility function of
  a diffusion.
\newblock \emph{J. Econometrics}, 184\penalty0 (1):\penalty0 124--144, 2015.

\bibitem[Cont and Tankov(2004)]{cont_financial_2004}
Cont  R and Tankov  P.
\newblock \emph{Financial modelling with jump processes}.
\newblock Chapman \& Hall/CRC Financial Mathematics Series. Chapman \&
  Hall/CRC, Boca Raton, FL, 2004.

\bibitem[Corradi and White(1999)]{corradi_specification_1999}
Corradi  V and White  H.
\newblock Specification tests for the variance of a diffusion.
\newblock \emph{J. Time Ser. Anal.}, 20\penalty0 (3):\penalty0 253--270, 1999.

\bibitem[Dette and Podolskij(2008)]{dette_testing_2008}
Dette  H and Podolskij  M.
\newblock Testing the parametric form of the volatility in continuous time
  diffusion models---a stochastic process approach.
\newblock \emph{J. Econometrics}, 143\penalty0 (1):\penalty0 56--73, 2008.

\bibitem[Dette and von Lieres~und Wilkau(2003)]{dette_test_2003}
Dette  H and von Lieres und Wilkau  C.
\newblock On a test for a parametric form of volatility in continuous time
  financial models.
\newblock \emph{Finance Stoch.}, 7\penalty0 (3):\penalty0 363--384, 2003.

\bibitem[Dette et~al.(2006)Dette, Podolskij, and Vetter]{dette_estimation_2006}
Dette  H, Podolskij  M, and Vetter  M.
\newblock Estimation of integrated volatility in continuous-time financial
  models with applications to goodness-of-fit testing.
\newblock \emph{Scand. J. Statist.}, 33\penalty0 (2):\penalty0 259--278, 2006.

\bibitem[Donoho et~al.(1995)Donoho, Johnstone, Kerkyacharian, and
  Picard]{donoho_wavelet_1995}
Donoho  D L, Johnstone  I M, Kerkyacharian  G, and Picard  D.
\newblock Wavelet shrinkage: asymptopia?
\newblock \emph{J. Roy. Statist. Soc. Ser. B}, 57\penalty0 (2):\penalty0
  301--369, 1995.

\bibitem[{Gonz\'alez-Manteiga} and
  Crujeiras(2013)]{gonzalez-manteiga_updated_2013}
{Gonz\'alez-Manteiga}  W and Crujeiras  R M.
\newblock An updated review of goodness-of-fit tests for regression models.
\newblock \emph{TEST}, 22\penalty0 (3):\penalty0 361--411, 2013.

\bibitem[Hoffmann et~al.(2012)Hoffmann, Munk, and
  {Schmidt-Hieber}]{hoffmann_adaptive_2010}
Hoffmann  M, Munk  A, and {Schmidt-Hieber}  J.
\newblock Adaptive wavelet estimation of the diffusion coefficient under
  additive error measurements.
\newblock \emph{Ann. Inst. Henri Poincar\'e Probab. Stat.}, 48\penalty0
  (4):\penalty0 1186--1216, 2012.

\bibitem[Ingster and Suslina(2003)]{ingster_nonparametric_2003}
Ingster  Y I and Suslina  I A.
\newblock \emph{Nonparametric goodness-of-fit testing under {G}aussian models},
  volume 169 of \emph{Lecture Notes in Statistics}.
\newblock Springer-Verlag, New York, 2003.

\bibitem[Jacod and Rei{\ss}(2014)]{jacod_remark_2012}
Jacod  J and Rei{\ss}  M.
\newblock A remark on the rates of convergence for integrated volatility
  estimation in the presence of jumps.
\newblock \emph{Ann. Statist.}, 42\penalty0 (3):\penalty0 1131--1144, 2014.

\bibitem[Jacod et~al.(2009)Jacod, Li, Mykland, Podolskij, and
  Vetter]{jacod_microstructure_2009}
Jacod  J, Li  Y, Mykland  P A, Podolskij  M, and Vetter  M.
\newblock Microstructure noise in the continuous case: the pre-averaging
  approach.
\newblock \emph{Stochastic Process. Appl.}, 119\penalty0 (7):\penalty0
  2249--2276, 2009.

\bibitem[Kleinow(2002)]{kleinow_testing_2002}
Kleinow  T.
\newblock Testing the diffusion coefficient.
\newblock \emph{Humboldt University of Berlin, Interdisciplinary Research
  Project 373: Quantification and Simulation of Economic Processes}, 2002.

\bibitem[Lee(2006)]{lee_bickelrosenblatt_2006}
Lee  S.
\newblock The {B}ickel-{R}osenblatt test for diffusion processes.
\newblock \emph{Statist. Probab. Lett.}, 76\penalty0 (14):\penalty0 1494--1502,
  2006.

\bibitem[Lee and Wee(2008)]{lee_residual_2008}
Lee  S and Wee  I S.
\newblock Residual empirical process for diffusion processes.
\newblock \emph{J. Korean Math. Soc.}, 45\penalty0 (3):\penalty0 683--693,
  2008.

\bibitem[Mancini(2009)]{mancini_non_2006}
Mancini  C.
\newblock Non-parametric threshold estimation for models with stochastic
  diffusion coefficient and jumps.
\newblock \emph{Scand. J. Stat.}, 36\penalty0 (2):\penalty0 270--296, 2009.

\bibitem[Munk and {Schmidt-Hieber}(2010)]{munk_lower_2010}
Munk  A and {Schmidt-Hieber}  J.
\newblock Lower bounds for volatility estimation in microstructure noise
  models.
\newblock In \emph{Borrowing strength: theory powering applications---a
  {F}estschrift for {L}awrence {D}. {B}rown}, volume~6 of \emph{Inst. Math.
  Stat. Collect.}, pages 43--55. Inst. Math. Statist., Beachwood, OH, 2010.

\bibitem[Mykland(1995)]{mykland_embedding_1995}
Mykland  P A.
\newblock Embedding and asymptotic expansions for martingales.
\newblock \emph{Probab. Theory Related Fields}, 103\penalty0 (4):\penalty0
  475--492, 1995.

\bibitem[Mykland and Zhang(2006)]{mykland_anova_2006}
Mykland  P A and Zhang  L.
\newblock A{NOVA} for diffusions and {I}t\^o processes.
\newblock \emph{Ann. Statist.}, 34\penalty0 (4):\penalty0 1931--1963, 2006.

\bibitem[Nguyen(2010)]{nguyen_time-changed_2010}
Nguyen  C M.
\newblock Time-changed residual based tests for asset pricing models in
  continuous time.
\newblock \emph{Unpublished manuscript, Indiana University}, 2010.

\bibitem[Ob{\l}{\'o}j(2004)]{obloj_skorokhod_2004}
Ob{\l}{\'o}j  J.
\newblock The {S}korokhod embedding problem and its offspring.
\newblock \emph{Probab. Surv.}, 1:\penalty0 321--390, 2004.

\bibitem[Papanicolaou and Giesecke(2014)]{papanicolaou_variation-based_2014}
Papanicolaou  A and Giesecke  K.
\newblock Variation-based tests for volatility misspecification.
\newblock \emph{Available at SSRN 2450776}, 2014.

\bibitem[Papapantoleon(2008)]{papapantoleon_introduction_2008}
Papapantoleon  A.
\newblock An introduction to {L}\'evy processes with applications in finance.
\newblock \emph{arXiv preprint arXiv:0804.0482}, 2008.

\bibitem[Podolskij and Ziggel(2008)]{podolskij_range-based_2008}
Podolskij  M and Ziggel  D.
\newblock A range-based test for the parametric form of the volatility in
  diffusion models.
\newblock \emph{CREATES Research Paper}, 22, 2008.

\bibitem[Rei{\ss}(2011)]{reis_asymptotic_2011}
Rei{\ss}  M.
\newblock Asymptotic equivalence for inference on the volatility from noisy
  observations.
\newblock \emph{Ann. Statist.}, 39\penalty0 (2):\penalty0 772--802, 2011.

\bibitem[Rei{\ss} et~al.(2014)Rei{\ss}, Todorov, and
  Tauchen]{reis_nonparametric_2014}
Rei{\ss}  M, Todorov  V, and Tauchen  G.
\newblock Nonparametric test for a constant beta over a fixed time interval.
\newblock \emph{arXiv preprint arXiv:1403.0349}, 2014.

\bibitem[Vetter(2012)]{vetter_estimation_2012}
Vetter  M.
\newblock Estimation of integrated volatility of volatility with applications
  to goodness-of-fit testing.
\newblock \emph{arXiv preprint arXiv:1206.5761}, 2012.

\bibitem[Vetter and Dette(2012)]{vetter_model_2012}
Vetter  M and Dette  H.
\newblock Model checks for the volatility under microstructure noise.
\newblock \emph{Bernoulli}, 18\penalty0 (4):\penalty0 1421--1447, 2012.

\end{thebibliography}
